\documentclass[12pt,english]{article}

\usepackage{lmodern}
\usepackage[T1]{fontenc}
\usepackage[latin9]{inputenc}
\usepackage{babel}
\usepackage{csquotes}
\usepackage{a4wide}

\usepackage[footnotesize]{caption}
\usepackage{subcaption}
\usepackage{import}

\usepackage[pdftex,dvipsnames]{xcolor}

\usepackage{enumerate}


\usepackage{microtype}
\makeatletter
\renewcommand\subparagraph{%
	\@startsection {subparagraph}{5}{\z@ }{3.25ex \@plus 1ex
		\@minus .2ex}{-1em}{\normalfont \normalsize \itshape }}%
\makeatother
\usepackage[
colorlinks=true,
linkcolor=blue,
citecolor=Green]{hyperref}

\usepackage[
tmargin=3cm,
bmargin=3cm,
lmargin=3.5cm,
rmargin=3.5cm]{geometry}

\usepackage[
style=authoryear,
firstinits=true,
style=numeric-comp,
sorting=none,
url=false,
doi=false,
isbn=false,
backend=biber]{biblatex}
\usepackage{dsfont}
\usepackage{verbatim}
\usepackage{mathrsfs}
\usepackage{amsmath}
\usepackage{dsfont}

\usepackage{amsthm}
\usepackage{amssymb}
\usepackage{iftex}
\usepackage{mathtools}

\providecommand{\corollaryname}{Corollary}
\providecommand{\definitionname}{Definition}
\providecommand{\lemmaname}{Lemma}
\providecommand{\remarkname}{Remark}
\providecommand{\theoremname}{Theorem}
\providecommand{\propositionname}{Proposition}
\providecommand{\examplename}{Example}
\providecommand{\settingname}{Setting}

\theoremstyle{plain}
\newtheorem{thm}{\protect\theoremname}[section]
\newtheorem{cor}[thm]{\protect\corollaryname}
\newtheorem{lem}[thm]{\protect\lemmaname}
\newtheorem{prop}[thm]{\propositionname}

\theoremstyle{definition}
\newtheorem{defi}[thm]{\protect\definitionname}
\newtheorem{expl}[thm]{\examplename}

\theoremstyle{remark}
\newtheorem{rem}[thm]{\protect\remarkname}

\numberwithin{equation}{section}

\usepackage{graphicx}
\usepackage{scalerel}

\newcommand{\outside}{c}%
\newcommand{\half}{\frac{1}{2}}

\newcommand{\vpos}{\vphantom{\os_{i}}}	%
\newcommand{\vpsum}{\vphantom{\sum}}	%

\newcommand{\ol}{\Omega}			%
\newcommand{\olt}{\Theta}
\newcommand{\olm}{\Theta}	
\newcommand{\os}{\omega}			%
\newcommand{\Rd}{\mathbb{R}^{d}} 	%
\newcommand{\RR}{\mathbb{R}} 		%
\newcommand{\cheese}{\Omega}		%

\newcommand{\bos}{\partial\os}		%
\newcommand{\bol}{\partial\ol}
\newcommand{\bolm}{\partial\olm}	
\newcommand{\bolt}{\partial\olt}		%
\newcommand{\bcheese}{\partial\cheese}	%
\newcommand{\supp}[1]{\mathrm{supp}(#1)}
\newcommand{\sli}[1]{S\left(#1\right)}				%

\newcommand{\measure}[1]{\;d\mu\negmedspace\left(#1\right)}%

\newcommand{\lt}[1]{L^{2}(#1)}					%
\newcommand{\ltd}[1]{L^{2}(#1;\Rd)}				%
\newcommand{\loct}[1]{L_{\mathrm{loc}}^{2}(#1)}		%
\newcommand{\sob}[1]{H^{1}(#1)}						%
\newcommand{\sobo}[1]{H^{1}_{0}(#1)}				%
\newcommand{\sobs}[2]{H^{#1}(#2)}				%
\newcommand{\sobdivo}[1]{H_{\mathrm{div},0}(#1)}		%
\newcommand{\sobloc}[2]{\mathring{H}^{#1}(#2)}	%
\newcommand{\sobloco}[2]{\mathring{H}^{#1}_0(#2)}	%
\newcommand{\harm}[1]{\mathscr{H}(#1)}				%
\newcommand{\harmo}[1]{\mathscr{H}(#1)}			%
\newcommand{\nulls}{\mathscr{N}}					%

\newcommand{\p}{\operatorname{\textit{P}}}			%
\newcommand{\s}{\operatorname{\textit{S}}}
\newcommand{\dlay}{K}						%
\newcommand{\tr}{tr}							%
\newcommand{\divv}[1]{\mathrm{\mathop{div}}(#1)}		%
\newcommand{\pd}[2]{\frac{\partial #1}{\partial #2}}		%

\newcommand{\funf}{\mathbf{f}}
\newcommand{\fung}{\mathbf{g}}

\newcommand{\vect}{\mathbf{v}}

\renewcommand{\P}{\mathcal{P}}%
\newcommand{\N}{\mathcal{N}}%
\newcommand{\R}{\mathcal{R}}%
\newcommand{\BB}{\mathbb{B}}%
\newcommand{\B}{\mathbf{B}}%
\renewcommand{\S}{\mathscr{S}}%
\newcommand{\m}{\mathbf{m}}%
\newcommand{\f}{\mathbf{f}}%
\newcommand{\g}{\mathbf{g}}%
\newcommand{\D}{\mathcal{D}}

\newcommand{\ot}{\leftarrow}

\usepackage{authblk}

\title{\textbf{\large Unique reconstruction of simple magnetizations from their magnetic potential}}

\author[1]{L. Baratchart}
\author[2]{C. Gerhards}
\author[2]{A. Kegeles\thanks{alexander.kegeles@geophysik.tu-freiberg.de}}
\author[2]{P. Menzel}

\affil[1]{
INRIA, Project FACTAS, 2004 route des Lucioles, BP 93, \linebreak 
Sophia-Antipolis 06902 Cedex, France \linebreak
}
\affil[2]{
TU Bergakademie Freiberg, Institute of Geophysics and Geoinformatics, \linebreak Gustav-Zeuner-Str. 12, 09599 Freiberg, Germany
}

\date{\vspace{-5ex}}

\begin{document}
\maketitle

\begin{abstract}
  Inverse problems arising in (geo)magnetism are typically ill-posed, in particular {they exhibit non-uniqueness}. Nevertheless, there exist {nontrivial model spaces}
  on which the problem is uniquely solvable. {Our goal is
    here to describe}  such spaces that {accommodate} constraints
  {suited for} applications. In this paper we treat the inverse
  {magnetization problem on a Lipschitz domain with}  fairly general topology. We characterize the subspace of $L^{2}$-vector fields that causes non-uniqueness, and identify a subspace of harmonic gradients on which the inversion becomes unique. This classification {has consequences} for applications and we present some of them in {the context} of geo-sciences. In the second part of the paper, we discuss the space of piecewise constant vector fields. This vector space is too large to make the inversion unique. But as we show, it contains a dense subspace in $L^2$ on which the problem becomes uniquely solvable, i.e., magnetizations from this subspace are uniquely determined by their magnetic potential. 
\end{abstract}

\section{Introduction}%
The goal of magnetic inverse problems is to recover the magnetization of an object from the magnetic field it creates. This problem is ill-posed as soon as the class of admissible magnetizations is large enough \cite{backus96,blakely95,runcorn75}, i.e., if it contains different magnetizations that create the same magnetic field. A central issue is then to restrict the class of {models so that the problem} has a unique solution. 

In many situations of interest, {one} can constrain admissible magnetizations by {making} a priori assumptions on the system. 
For example, in large-scale lithospheric studies, one may confine the magnetization of the lithosphere to the Earth's surface (e.g., \cite{gubbins11,masterton13}). In estimations of the thickness of a magnetized layer, one often assumes the thickness to be constant (though unknown) and the magnetization to be self-similar (e.g., \cite{blakely95,boul09,maus97} and references therein). When studying rock samples with magnetic inclusions one might assume the magnetization to be constant within each inclusion \cite{degroot18}. 
In {these examples,} the additional information may or may not suffice to render the inversion well-posed. To {understand when it happens}, we need to {characterize classes of models over} which the inversion is possible.

In gravimetry, related uniqueness questions have been studied in great detail; see, e.g., \cite{michel05,michel08,isakov90,isaqian13,sanso12} to name only a few recent publications. Though similar in nature, the magnetic inverse problem is more challenging due to its vectorial setup, meaning that combinations of the vector components of magnetizations can lead to cancellations that cannot occur for scalar densities in gravimetry. As a result, the details of the magnetic inverse problem are less understood.

Recently, a new interest for uniqueness issues arose in applications to lithospheric studies \cite{gubbins11,verles18} and in paleo-magnetism \cite{degroot18,lima13,vervelidou16}. When the magnetization is confined to an infinitely thin surface, like a sphere, a thin plate, or a general Lipschitz surface, several (partial) uniqueness results have been discussed in \cite{bargerkeg20,baratchart13,gerhards16a,baratchartgerhards16,gerhards16b,maushaak03,verles19}. For a {more} general problem (without restricting the magnetization to {necessarily lie on a surface),  TV regularization for magnetizations modeled by measures was} discussed in  \cite{barvilhar19}. A specific setup for volumetric magnetizations in terms of \lq\lq ideal bodies'' has been treated, e.g., in \cite{blakely95,parker91}. Furthermore, a somewhat related inverse problem (aiming for the current and not for the magnetization) is discussed in \cite{michel18} and earlier references therein, however, only for domains that are balls. 

In this paper, we address the uniqueness {issue for a fairly general 	magnetic inverse problem with volumetric samples}. 
For simplicity of the presentation, we will consider the magnetic potential instead of the magnetic field, which for most applications {yields}
an equivalent formulation of the problem (see Remark \ref{rem:invisible}). 
We assume {magnetizations to be of $L^{2}$-class on a Lipschitz domain $\ol$ with  fairly general topology (see Section \ref{sec:aux}), and discuss} two subspaces of $L^{2}$-vector fields that yield a unique inversion.
{The first comprises harmonic gradients, and the second comprises
	fields constant on rectangular parallelepipeds. The former is
	a closed subspace of $L^2$-vector fields,
	and the latter a dense subspace thereof.}

In Section \ref{sec:nullconnect},
we {investigate the case where} the magnetic potential created by a magnetization { on $\ol$ is only  known over} some (not necessarily all) connected components $\olt$ of the exterior $\Rd\setminus\overline{\ol}$.
{The} main result {is} Theorem \ref{thm:null_spaces}, {that parametrizes those
	magnetizations creating the zero  magnetic potential in $\olt$, and
	shows in particular} that only the harmonic gradient {component of the Hodge decomposition of a magnetization can
	produce a non-vanishing magnetic potential.}
Based on this characterization, we derive corollaries that {address} some of the geophysical constraints mentioned earlier. In Corollary \ref{cor:cluster}{,} we state that the magnetic potentials of a collection of {separated} magnetized grains cannot cancel each other out. This corollary generalizes the main result from \cite{fabian19} to Lipschitz domains{, and is a consequence of the fact that an invisible magnetization supported on disjoint closed sets has invisible restriction to each of them \cite[Lem. 2.4]{barvilhar19}}. In Corollary \ref{cor:unidirect}{, we further show that if each grain is magnetized
	in a single direction then one can uniquely reconstruct these directions
	provided the} net moment of the magnetization is nonzero. Our final corollary {in} that section (Corollary \ref{cor:suppnull}) implies that it is impossible to deduce the thickness of the magnetized layer of the lithosphere purely from {knowledge of} the magnetic field of the Earth.

{As we said},
Section \ref{sec:nullconnect} {characterizes completely solutions to the inverse problem that are harmonic gradients}. Nevertheless, from the geophysicist's point of view harmonic gradients are rather unnatural. For one thing, harmonic functions cannot vanish on an open domain unless {they are}
identically zero, whereas realistic magnetizations of rock samples can. For another thing, many geophysical applications assume the magnetization to be piecewise constant. Such vector fields are never harmonic {gradients
	unless they are constant}.

In section \ref{sec:unique}, we {focus on piecewise constant vector fields on finitely many rectangular parallelepipeds, that} we call {box-simple vector fields}. Our main result {there is} Theorem \ref{thm:simple-funtions}, {asserting}
that box-simple vector fields are uniquely determined by their magnetic potential. Put differently, the inverse problem is uniquely solvable on the vector space of box-simple vector fields {(if solvable at all in this class). Let us stress that the result does not require {a priori} knowledge on the number, size nor location of the  parallelepipeds on which the field is constant}. To our knowledge, this is the first uniqueness result of this kind. {In fact, it is natural to discretize $L^2$- magnetizations as piecewise constant vector fields on cubes, and the present study also sheds light on the regularizing effect of such a discretization:  Theorem \ref{thm:simple-funtions} entails that the discretized problem is well-posed, but one can surmise it is ill-conditioned. The fact that stability estimates blow up exponentially with the mesh of the cubes is discussed on examples in Section \ref{subsec:sb}}

{The necessary notation, as well as auxiliary results,} are gathered in section \ref{sec:aux}. Throughout the paper{,} we provide some basic examples and try to motivate connections {to  geophysically} relevant situations.

\section{Notation and auxiliar{y} results}\label{sec:aux}

\paragraph{Lipschitz domain{s}.}
An open {set} is Lipschitz if its boundary is locally {a Lipschitz graph; that is: for each boundary point $x$
	and after an appropriate rigid motion (that may depend on $x$),
	the boundary becomes the graph of a Lipschitz function in a neighborhood of $x$} (for more details see for example \cite[p. 83 (4.9)]{adams03}). We do not assume that a Lipschitz {open set} is bounded{,} neither that its boundary is connected. However, we will always assume that the boundary of a Lipschitz {open set} is compact. {Of necessity, the connected components of a Lipschitz open set are finite in number and positively separated:  for if there was a sequence of boundary points no two
	members of which belong to the same connected component of the boundary, then the local Lipschitz graph condition would be violated in every accumulation point of that sequence}. Almost everywhere on {the boundary
	of a Lipschitz open set,} there exists a {unit outwards pointing} normal field $\nu$. {As is customary, we often call a connected open set
	a {\em domain},
	and so we speak of Lipschitz domains to mean connected Lipschitz open sets.}

We single out two special types of Lipschitz open sets that are most useful for applications, we call them {\textit{Lipschitz cluster}}
and {\textit{Lipschitz cheese} respectively}
(see Figure \ref{fig:graincheese} for an example).

A Lipschitz cluster is a Lipschitz {open set} that is a union of {finitely many disjoint, bounded, and simply connected Lipschitz domains}. A Lipschitz cluster is {appropriate to specify} a collection of magnetized grains within an otherwise unmagnetized material; for example, a rock sample with magnetic inclusions, or a human brain with activated local regions.

A Lipschitz cheese is a bounded {and simply connected Lipschitz domain} with a Lipschitz cluster removed, i.e.,  a bounded finitely connected Lipschitz domain.
A Lipschitz cheese is {appropriate to describe} a magnetized material with unmagnetized inclusions, e.g., the lithosphere of the Earth, or a magnetic rock sample with cavities. 
\begin{figure}
	\centering
	\def\svgwidth{12cm}
\begingroup%
  \makeatletter%
  \providecommand\color[2][]{%
    \errmessage{(Inkscape) Color is used for the text in Inkscape, but the package 'color.sty' is not loaded}%
    \renewcommand\color[2][]{}%
  }%
  \providecommand\transparent[1]{%
    \errmessage{(Inkscape) Transparency is used (non-zero) for the text in Inkscape, but the package 'transparent.sty' is not loaded}%
    \renewcommand\transparent[1]{}%
  }%
  \providecommand\rotatebox[2]{#2}%
  \newcommand*\fsize{\dimexpr\f@size pt\relax}%
  \newcommand*\lineheight[1]{\fontsize{\fsize}{#1\fsize}\selectfont}%
  \ifx\svgwidth\undefined%
    \setlength{\unitlength}{221.61597157bp}%
    \ifx\svgscale\undefined%
      \relax%
    \else%
      \setlength{\unitlength}{\unitlength * \real{\svgscale}}%
    \fi%
  \else%
    \setlength{\unitlength}{\svgwidth}%
  \fi%
  \global\let\svgwidth\undefined%
  \global\let\svgscale\undefined%
  \makeatother%
  \begin{picture}(1,0.51703755)%
    \lineheight{1}%
    \setlength\tabcolsep{0pt}%
    \put(0,0){\includegraphics[width=\unitlength,page=1]{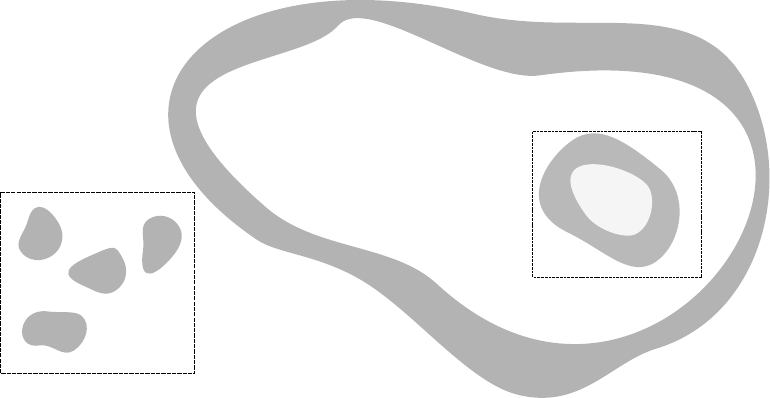}}%
    \put(0.87116897,0.16859955){\color[rgb]{0,0,0}\makebox(0,0)[lt]{\lineheight{1.25}\smash{\begin{tabular}[t]{l}\footnotesize C\end{tabular}}}}%
    \put(0,0){\includegraphics[width=\unitlength,page=2]{lipschitz.pdf}}%
    \put(0.55052097,0.24154192){\color[rgb]{0,0,0}\makebox(0,0)[lt]{\lineheight{1.25}\smash{\begin{tabular}[t]{l}\footnotesize B\end{tabular}}}}%
    \put(0.20424301,0.05737583){\color[rgb]{0,0,0}\makebox(0,0)[lt]{\lineheight{1.25}\smash{\begin{tabular}[t]{l}\footnotesize A\end{tabular}}}}%
  \end{picture}%
\endgroup%

	\caption{The union of grey regions is a Lipschitz domain. In particular it contains the Lipschitz cluster (region A), the Lipschitz cheese (region B), and the special case of the Lipschitz cheese with a single cavity (region C).}
	\label{fig:graincheese}
\end{figure}

To {denote locally constant functions on a
	Lipschitz open set $\Omega$ or on its boundary $\partial\Omega$,} we use the following notation: let $\chi_A$ be the characteristic function {which is $1$ on the set  $A\subset\Rd$ and 0 elsewhere}. For a Lipschitz
{open set $\ol$, we define the spaces}
\begin{align*}
	\RR_{\ol} 	&= \mathrm{span}_{\RR} \{ \chi_{\os_{k}} \ : \ \os_{k} \text{ a connected component of } \ol\}, \\
	\RR_{\bol} 		&= \mathrm{span}_{\RR} \{ \chi_{\bos_{k}} \ : \ \os_{k} \text{ a connected component of } \ol\}.
\end{align*}
If $A$ is a set in $\Rd$, we will denote {by $A^{\outside}$
	the complement $\Rd \setminus \overline{A}$.}

\paragraph{Function spaces.}
The spaces we use in this paper are standard and their detailed definition may be found in {many textbooks} (for example \cite[chap. 3]{adams03} or \cite[chap. 3]{mclean2000strongly}).
Unless {otherwise stated,} we assume for the rest of this section that $\ol$ is a Lipschitz {open set} in $\Rd$ ($d\geq3$). 

\subparagraph{Sobolev spaces.}
The spaces $\ltd \ol$ and $\ltd \bol$ denote the usual { spaces of square integrable} $\RR^{d}$-valued vector fields. The {norms
	are induced by the scalar products}
\begin{align}
	\left\langle \funf, \fung \right\rangle_{\ltd \ol} = \int_{\ol} \funf(x) \cdot \fung(x) \measure{x}
	&&
	\funf, \fung\in\ltd \ol, \\
	\left\langle \funf, \fung \right\rangle_{\ltd {\partial\ol}} = \int_{\partial\ol} \funf(x) \cdot \fung(x) \,\,d\sigma(x)
	&&
	\funf, \fung\in\ltd {\partial\ol},
\end{align}
where $\cdot$ denotes the Euclidean scalar product in $\Rd$ and $\mu$ is
{$d$-dimensional Lebesgue measure while $\sigma$
	is $(d-1)$-dimensional Hausdorff-measure.} 
For {scalar-valued functions,} we simply write $\lt \ol$ or $\lt \bol$.  To distinguish {scalar-valued} functions from vector fields{, we  use} bold-face font for the latter. The space $\loct \ol$ will {designate} locally square integrable functions. 

With $\sobs{1}{\ol}$ we denote the Sobolev space of square integrable functions with square integrable gradients. {By $\sobs{s}{\bol}$ $(0\leq s \leq 1)$, we indicate the standard scale of Hilbertian} Sobolev spaces on the boundary of a Lipschitz domain (for details see for example \cite[p. 96 ff.]{mclean2000strongly}); {in fact we will only need} $s=1/2$. We write $\sobs{-s}{\bol}$ for the topological dual of $\sobs{s}{\bol}$  and denote the
{duality} pairing between $f\in \sobs{-s}{\bol}$ and $g\in\sobs{s}{\bol}$  by $\left(f,g\right)$. We denote the support of a function $f$ by $\supp{f}$.

As usual, $C^{\infty}_{0}(\ol)$ denotes the space of smooth functions with compact support in $\ol$. Subsequently, we let
	\begin{equation}
		\sobdivo \ol		= L^{2} \text{-closure of } \left\{	\funf \ : \ \funf \in C^{\infty}_0(\ol,\Rd), \ \divv \funf  = 0\right\}{,}
	\end{equation}
{to be interpreted as} the space of divergence-free fields with vanishing normal {at the boundary of $\Omega$.} We will also use the homogeneous Sobolev spaces
\begin{align*}
	\sobloc{1}{\ol} 	&= \left\{ f \ : \ f \in \loct \ol, \ \nabla f \in \ltd \ol	\right\}, \\
	\sobloco{1}{\ol} 	&= \text{closure of } C^{\infty}_{0} (\ol) \text{ in the norm } \| \nabla \cdot \|_{\ltd \ol},
\end{align*}
It is clear from the definition that the space of gradients of $\sobloco{1}{\ol}${-functions}, denoted $\nabla \sobloco{1}{\ol}$, is closed in $\ltd \ol$. %
When $\ol$ is bounded, it follows from the Poincar\'e inequality that $\sobloc{1}{\ol}= \sobs{1}{\ol}$.

\subparagraph{Harmonic Sobolev space.}
We write $\harmo \ol$ {for the space of harmonic functions on $\ol$ that lie} in $\sobloc{1}{\ol}$ and {vanish} at infinity (if $\ol$ is unbounded). Then{,} $\nabla \harmo \ol$ will denote the space of gradients of functions in $\harmo \ol$. {It is a closed subspace of $L^2(\Omega;\mathbb{R}^d)$, 	however functions in $\harmo\ol$ need not belong to $H^1(\ol)$ when $\ol$ is unbounded. For example, if we let $\ol:=\mathbb{R}^d\setminus \overline{\mathbb{B}_R(0)}$ then $x\mapsto1/|x|$ lies in $\harmo\ol$.}

\paragraph{The trace and the normal derivative.} 
{Below, we review boundary traces of Sobolev functions and
	of gradients of harmonic functions.} %

\subparagraph{Trace.}
For a Lipschitz domain $\ol${, we denote the trace operator
	on  $\bol$} by $\tr_{\bol}$. It is a bounded linear operator 
\begin{align}
	\tr_{\bol}\colon\sobs{1}{\ol}\to\sobs{\half}{\bol}
\end{align}
that acts on smooth functions by restriction{:} $\tr_{\bol}(f) = f\vert_{\bol}$ when $f$ is in $C^{\infty}(\overline{\ol})$. Since the trace of a function depends only on its  {behaviour} near the boundary, {the trace operator extends  continuously} to $\sobloc{1}{\ol}$ (and thus to $\harmo \ol$). For every {$f\in\sobloco{1}{\ol}$ one has} $\tr_{\bol}(f)=0$. {If $\Omega$ is bounded then the converse holds, however when $\ol$ is unbounded  the subspace of functions with
  zero trace in $\sobloc{1}{\ol}$ can be strictly larger than $\sobloco{1}{\ol}$,} see for example \cite[thm. 2.15 and cor. 2.17]{Simader:1996aa}.
{An important feature of the trace operator is that it is surjective: if $h\in H^{1/2}(\partial\Omega)$ there is $g\in \sobloc{1}{\ol}$ with
  $\tr_{\bol}(g)=h$ (we may even choose $g$ to be harmonic, see discussion 
after \eqref{BDLO}).}
\subparagraph{Normal derivative.}
{Every $f\in\harmo \ol$ has a well{-}defined normal derivative $\partial f/\partial\nu\in\sobs{-\half}{\bol}$, {acting on $H^{1/2}(\partial\Omega)\ni h=\tr_{\bol}(g)$ according to the rule }
\begin{align}
		\label{intpp}
		\left(	\pd{f}{\nu} ,\tr_{\bol}(g)\right) 
		= \left\langle \nabla f , \nabla g\right\rangle_{\ltd \ol},
		&&
		g \in \sobloc{1}{\ol}{,}
	\end{align}
        {where we note that the right hand side of \eqref{intpp}
          indeed depends only on $\tr_{\bol}(g)$ by \cite[thm. 2.4]{Simader:1996aa} and the Schwarz inequality.}
          We will refer to the formula \eqref{intpp}, as  to \enquote{integration by parts}.}

\paragraph{Potentials.} {Single and double layer potentials 
	have been studied extensively in potential theory.  For details, we refer the reader to the sizeable literature on this topic} (for example \cite{wermer74,ver84, dibenedetto95, fabmenmit98, Mitrea:1997aa,mitrea09}). Below, we set up notation and discuss a couple of points regarding unbounded domains, which are rarely addressed in the literature. Let $\omega_{d}$ be the area of {the} unit sphere in $\Rd$ and {put} $C_{d}=(d-2)\omega_{d}${,} so that the fundamental solution of the Laplacian for $d\geq 3$ {is given by}
\begin{align*}
	N(x) = \frac{1}{C_{d}} \frac{1}{\vert x \vert^{d-2}}
	\qquad
	\left(x \in \RR^{d}\setminus\{0\}\right).
\end{align*}
Recall that it is harmonic everywhere except at the origin.

\subparagraph{Magnetic potential.}
For $\funf \in \ltd \ol$ we define the magnetic potential of $\funf$ by
\begin{align*}
	\p_{\ol^{\outside}\ot\ol} \left(	\funf	\right) (x) 
	= \int_{\ol} \nabla N(x-y) \cdot \funf(y) \ \measure{y} 
	= \left\langle \nabla N(x - \cdot), \funf \right\rangle_{\ltd \ol}, 
	\quad
	\left(x \in \ol^{\outside}\right).
\end{align*}
The {rationale behind the cumbersome subscript} is as follows: the {origin} of the arrow (here $\ol$) denotes the {set
	over which we integrate; the arrowhead} (here $\ol^{\outside}$) denotes  the region {where we evaluate the magnetic potential}. That is{:} $P_{\ol^{\outside}\ot\ol}$ maps a function defined on $\ol$ to a function defined on $\ol^{\outside}$. If the support of the field $\funf$ is smaller than the integration region{,} we implicitly extend $\funf$ by zero to match the integration domain. If the support of $\funf$ is larger than the integration region{,} we implicitly restrict $\funf$ to the integration domain.  Even though this {notation convention} is not a common one, it will turn out {to be} useful. {Of course, the integral defining $P_{\ol^{\outside}\ot\ol}(x)$ makes good sense
	for $x\in\Omega$ as well, since it  is absolutely convergent for $\mu$-a.e.
	$x$, compare \cite[Prop. 214]{barvilhar19}. In fact, the magnetic potential is well-defined as a member of $\sobloc{1}{\mathbb{R}^d}$, but we shall not use this here.}

\subparagraph{Single layer potential.}
The single layer potential of $f\in\sobs{-\half}{\bol}$  is {the} harmonic function in $\Rd\setminus\bol$ defined by
\begin{align*}
	\s_{\Rd\setminus\bol\ot\bol} (f) (x) = \int_{\bol} N(x-y) f(y) \measure{y}
	= \left(f, N(x-\cdot)\right),
	\qquad
	\left(x\in \Rd\setminus\bol\right)
\end{align*}
The {rationale behind} the subscript is the same as above. The 
{single layer potential $\s_{\Omega\ot\bol}$
	(resp. $\s_{\Omega^c\ot\bol}$) maps $H^{-1/2}(\partial\ol)$ into
	$\sobloc{1}{\ol}$ (resp. $\sobloc{1}{\Omega^c}$) continuously; if
	$\Omega$ is bounded this follows from  \cite[thm. 3.1]{fabmenmit98}, and 
	when $\Omega$ is unbounded we get the result from the bounded case and a direct estimate of the derivative at $x$ for large $|x|$.
	Moreover, $\tr_{\bol}(\s_{\Omega\ot\bol}f)=\tr_{\bol}(\s_{\Omega^c\ot\bol}f)$,
	as follows by density from the case where $f\in L^2(\partial\Omega)$ which is
	standard \cite[lem. 1.8]{ver84}.
	Thus, the trace of the single layer potential
	on $\partial\Omega$ defines} a bounded linear operator:
\begin{align}
	\s_{\bol} \colon \sobs{-\half}{\bol} \to \sobs{\half}{\bol}.
\end{align}
This operator is invertible \cite[thm. 8.1 (5)]{fabmenmit98} and symmetric. 
{The latter means} that for every $f\in\sobs{-\half}{\bol} $ and $g\in\sobs{\half}{\bol}${,} we have 
\begin{align}
	\left\langle\s_{\bol}\left(f\right),g\right\rangle_{\lt \bol} 
	= \left( f, \s_{\bol}\left(g\right) \right){,}
\end{align}
{where we observe, since $H^{1/2}(\partial\Omega)\subset L^2(\partial\Omega)$,  that indeed $\s_{\bol}\left(g\right)\in H^{1/2}(\partial\Omega)$ for it even belongs to the smaller Sobolev space
	$W^{1,2}(\partial\Omega)$ that we did not introduce, 
	see \cite[lem. 1.8]{ver84}.}
\subparagraph{Double layer potential.}
\label{DLPS}
{The double layer potential of $f\in H^{1/2}(\partial\Omega)$
	is the harmonic function defined as}
\begin{align}
	{\mathcal{K} (f) (x) = \int_{\bol} \nabla N(y - x) \cdot \nu (y) \  f(y) {d\sigma(y)},\qquad x\in\mathbb{R}^d\setminus\partial\Omega.}
\end{align}
{The  boundary double layer potential is 
	the singular integral  operator:}
\begin{align}
  \dlay (f) (x) = p.v. \int_{\bol} \nabla N(y - x) \cdot \nu (y) \  f(y) {d\sigma(y),\qquad x\in\partial\Omega}.
  \label{BDLO}
\end{align}
{The latter is bounded on $\sobs{\half}{\bol}$, and the trace of $K(f)|_{\Omega}$ (resp. $K(f)|_{\mathbb{R}^d\setminus\overline{\Omega}}$) on $\partial \Omega$ is 	$\frac{1}{2}I+\dlay$ (resp. $-\frac{1}{2}I+\dlay$).}
{  Here and below, $I$  denotes the identity operator on
	whichever   functional space is involved.}

{The double layer potential is a classical tool to solve the Dirichlet problem for the Laplacian in a bounded Lipschitz domain $\Omega$:
	for $g\in H^{1/2}(\partial \Omega)$, there is a unique $u\in H^1(\Omega)$
	with $\tr_{\bol}u=g$ which is given by $\mathcal{K}(\frac{1}{2}I+\dlay)^{-1}g$.}
For {this and more on Sobolev and Besov spaces
	on which $\pm\frac{1}{2}I+\dlay$ is invertible,
	see \cite[thm. 8.1]{fabmenmit98}. When $\Omega$ is unbounded, the Dirichlet problem with data in $H^{1/2}(\partial\Omega)$ can be solved uniquely
	in $\sobloc{1}{\ol}$; this follows easily by Kelvin transform
	from the result on bounded domains; see, e.g., \cite{Simader:1996aa} for the definition and properties of the Kelvin transform. However, one generally has to renounce membership of $u$ in $H^1(\Omega)$, in particular $u$ needs not be the double layer potential of a $H^{1/2}(\partial\Omega)$-function.}

{By $\dlay^{\star}$ we will denote the Banach space adjoint of $\dlay$, which is a bounded linear operator on $\sobs{-\half}{\bol}$.} The normal derivative of $\s_{\ol\ot\bol} (f)$ (resp. $\s_{\ol^c\ot\bol} (f)$) on $\partial\Omega$
	is $-\frac{1}{2}I+\dlay^{\star}$ (resp. $\frac{1}{2}I+\dlay^{\star}$). The single layer potential is
	a classical tool to solve the Neumann problem for the Laplacian in a bounded Lipschitz domain $\Omega$:
	for $g\in H^{-1/2}(\partial \Omega)$ satisfying the (necessary) compatibility condition
	$\left(g, c\right) =0$ for all $c \in \RR_{\bol}$, there is 
	$u\in H^1(\Omega)$, unique up to an additive constant,
	with $\partial u/\partial\nu=g$ which is given by $\mathcal{S}(-\frac{1}{2}I+\dlay^*)^{-1}g$; see see \cite[thm. 9.2]{fabmenmit98}.
{Actually, the same result holds when  $\Omega$ is unbounded
	and $u$ is required to vanish at infinity,
	except that $u\in\sobloc{1}{\ol}$ but $u\notin H^1(\Omega)$  in general.
	Indeed,  a solution is still given by the previous formula, and uniqueness up to a constant follows from an integration by parts that yields
	the identity $\|\nabla u\|^2_{L^2(\Omega,\mathbb{R}^d)}=(u,
	\partial u/\partial\nu)_{\partial\Omega}$.} 
For {this and more on Sobolev and Besov spaces
	on which $\pm\frac{1}{2}I+\dlay$ is invertible,
	see \cite[thm. 8.1]{fabmenmit98}.}

We {turn to a} result that we use throughout the paper. Even though it is known in much greater generality (see for example the Hodge decomposition \cite[chap. 2]{schwarz2006hodge}), {it is easy in the case at hand to give a proof that we include for the convenience of the reader.}

\begin{lem}[Helmholtz-Hodge-decomposition]\label{lem:helmholtz_harmonic}
	Let $\ol$ be a Lipschitz {open set} in $\Rd$ ($d\geq3$). The space $\ltd \ol$ is an orthogonal direct sum 
	\begin{align*}
		\ltd \ol	= \nabla \sobloco{1}{\ol} \oplus \nabla \harmo \ol \oplus \sobdivo \ol.
	\end{align*}
\end{lem}

\begin{proof}
	The well{-}known Helmholtz decomposition says that the space $\ltd \ol$ is an orthogonal direct sum (for example \cite{Simader:1992aa}) %
	{:}
	\begin{align}
		\ltd \ol	= \nabla \sobloc{1}{ \ol} \oplus \sobdivo \ol. \label{eqn:h2}
	\end{align}
	Thus, we only need to show that  $\nabla \sobloc{1}{ \ol} = \nabla \sobloco{1}{ \ol} \oplus \nabla \harmo \ol$.
	{For this, assume that $\nabla \varphi\in\nabla \sobloc{1}{\ol}$ is} orthogonal to $\nabla\sobloco{1}{\ol}$. In particular $\nabla \varphi$ is then orthogonal to gradients of smooth functions with compact support in $\ol$. Thence, $\varphi$ is weakly harmonic, and thus harmonic by Weyl's lemma. {Consequently, $\varphi\in \harmo \ol$ so that
		$\nabla\varphi\in\nabla \harmo \ol$ (cf. \cite[thm. 3.5]{Simader:1996aa}). Conversely, if $\varphi\in \harmo \ol$ and $g\in \sobloco{1}{\ol}$, we get from an} integration by parts and the {condition $\tr_{\bol}(g)=0$
		(which holds for functions in $\sobloco{1}{\ol}$) that $\left\langle \nabla f , \nabla g\right\rangle_{\ltd \ol}=0$.
		Consequently, $\nabla  \harmo \ol$ is the orthogonal space to 
		$\nabla \sobloco{1}{\ol}$ in $\sobloco{1}{\ol}$}.
	Since $\nabla \sobloco{1}{\ol}$ is closed in $\ltd \ol$, 
	{this achieves the proof}.
\end{proof}

\section{Harmonic gradients}\label{sec:nullconnect} 

In this section we identify the subspace of harmonic gradients on Lipschitz domains that can be uniquely recovered from their magnetic potential. In the rest of the paper we will denote the null-space of an operator $T$ by $\nulls[T]$ and use the following terminology.

\begin{defi}
	We say that a vector field $\funf\in\ltd \ol$  is \textit{invisible in $A \subset \ol^{\outside}$} if 
	\begin{align*}
		\p_{A\ot\supp{\funf}}\left(\funf\right) = 0.
	\end{align*}
	If $\funf$ is invisible in $\supp{\funf}^{\outside}$ then we say that $\funf$ is in the null-space of the magnetic potential or that it is simply \textit{invisible}. If $\funf$ is not invisible, we call it \textit{visible}.     
\end{defi}

\begin{rem}\label{rem:invisible}
	In the literature the term ``invisible'' sometimes refers to those magnetizations whose magnetic field vanishes, instead of the magnetic potential. If $\funf$ is invisible in $A\subset \Rd$ and $A$ is unbounded then both definitions coincide, since $\p_{A\ot\supp{\funf}} (\funf)$ {tends to} zero at infinity. But if $A$ is bounded then a constant magnetic potential will have a vanishing gradient, i.e., a vanishing magnetic field. We caution the reader to {take into account this slight} difference when comparing our result{s} to other sources in the literature{, for instance \cite{barlebnem21}}.
	Since { in practice data are typically available in the unbounded component of the complement} of the object of interest, this {difference should be of little significance in applications.}
\end{rem}

\subsection{Uniqueness for harmonic gradients}
{Hereafter we make the following convention: when
	$\Omega$ is a Lipschitz open set and $E$ a union of connected components of
	$\partial\Omega$, we consider $H^{1/2}(E)$ (resp. $H^{-1/2}(E)$) as a subspace of $H^{1/2}(\partial\Omega)$ (resp. $H^{-1/2}(E)$), using the extension
	by  zero of functions
	(resp. distributions) to $\partial\Omega\setminus E$.}
\begin{thm}\label{thm:null_spaces}
	Let $\cheese$ be a bounded Lipschitz {open set} in $\Rd$ ($d\geq3$){,} and let $\olt$ be an arbitrary collection of connected components of $\ol^{\outside}$. The space of {vector fields in $L^2(\Omega;\mathbb{R}^d)$ that are invisible in} $\olt$ is
	\begin{align}
		\nulls	\left[	\p_{ \olt\ot\cheese	\vpsum}	\right] 
		= \left( \vpsum
		\left. \nabla\harmo {\olt^{\outside}} \right|_{\cheese} 
		\right)^{\perp}
		= \nabla \sobo \cheese \oplus \sobdivo \cheese \oplus \nabla \mathscr{I}_{\olt} \label{eq:null-space-general}{,}
	\end{align}
	where {the superscript $\perp$ refers to the orthogonal complement in $L^2(\ol;\mathbb{R}^d)$ and} $\mathscr{I}_{\olm}$ is the space of harmonic functions $\varphi$ {in $\harmo{\ol}$} of the form
	\begin{align}
		\varphi =  \s_{\cheese\ot\bcheese} \left(-\frac{1}{2}I + \dlay^{\star}\right)^{-1}{\s_{\bcheese}}^{-1} \left(f\right) \label{eq:form-of-null-harmonics}
	\end{align}
	parametrized by {those} functions  $f\in\sobs{\half}{\bol\setminus\bolm}$ that satisfy
	\begin{align*}
		\left({f, \s_{\bol}}^{-1}\left(c\right)\right)
		=0 
		&& 
		\text{for every } c \in \RR_{\bol}.
	\end{align*}
\end{thm}

\begin{proof}
	{Clearly, vector fields in $L^2(\Omega;\mathbb{R}^d)$ that are invisible in $\olt$ form a closed vector space, by the dominated convergence theorem. Observe} that the {evaluation of the} magnetic potential {of $\funf\in L^2(\Omega:\mathbb{R}^d)$ in a point $x\in\Omega^c$ is the}  scalar product {of $\funf$ with the gradient of  the harmonic function $y\mapsto N(x-y)$ on $\Omega$.} Hence, the orthogonal complement of harmonic gradients, {which is  $\left(\nabla\harmo \cheese\right)^{\perp}=\nabla \sobloco{1} {\cheese} \oplus \sobdivo \cheese$ by Lemma
		\ref{lem:helmholtz_harmonic}}, is always invisible {in $\ol^{\outside}$ and  a fortiori in $\olt$}. 
	
	{Next, let us show that} those invisible fields in $\olt$
	{of the} form $\funf=\nabla \varphi$ for some $\varphi\in\harmo\cheese$ {match the space $\nabla \mathscr{I}_{\olt}$}. For
	{such} fields we can use integration by parts{,}  and rewrite the {vanishing of the} magnetic potential {for $x\in\olt$} as
	\begin{align}
		0	=\p_{{\olm}\ot\ol} \left(	\funf	\right) (x)
		= \left\langle\nabla N(x - \cdot ),\nabla \varphi \right\rangle_{\ltd \ol} 
		= \s_{{\olm}\ot\bol} \left(\pd{\varphi}{\eta}\right)(x).\label{eq:pot_as_single}
	\end{align}
	{By uniqueness of the solution to the Dirichlet problem
		in $H^1(O)$  with given trace in $H^{1/2}(\partial O)$ for a bounded Lipschitz domain $O$, %
		and uniqueness of the solution to the Dirichlet problem in $\sobloc{1}{A}$  with given trace in $H^{1/2}(\partial A)$ and limit zero  at infinity for an unbounded Lipschitz domain $A$ (having compact boundary), \eqref{eq:pot_as_single} is equivalent to say that the trace of $\s_{\olm\ot\bol} \left(\pd{\varphi}{\eta}\right)$ 
		vanishes on the subset $\bolt$ of $\partial\Omega$}, {while on the complementary part} $\bol\setminus\bolm$ the trace can be arbitrary. Therefore, we {may} parametrize the space of harmonic gradients {$\nabla\varphi$} invisible in $\olt$ by functions $f \in \sobs{\half}{\bol\setminus\bolm}$ via  the correspondence
	\begin{align}
		\s_{\bcheese}\left(\pd{\varphi}{\eta}\right) = f,
		\label{eq:condition-one}
	\end{align}
	where{, according to our convention,} we implicitly extend $f$ by zero to the whole {of} $\bol$. 
	Since {$\s_{\bcheese}:H^{-1/2}(\partial\Omega)\to H^{1/2}(\partial\Omega)$} is invertible by \cite[thm. 4.1 (9)]{Mitrea:1997aa}, we can reformulate this {correspondence as solving} a Neumann problem: for some $f\in\sobs{\half}{\bol\setminus\bolm}$ the function $\varphi$ has to {satisfy:}
	\begin{align}
		\begin{cases}
			\varphi 			\in \sob \ol \\
			\Delta \varphi 		= 0 & \text{on } \ol \\
			\partial_{\eta}\varphi 	= \s_{\bcheese}^{-1}(f) & \text{on } \bol.
		\end{cases}
		\label{eq:Neumann-prob}
	\end{align}
	A solution to this problem exists if and only if $f$ satisfies the
	{conditions}
	\begin{align}
		\left({\s_{\bol}}^{-1}(f), c\right) =0 
		=\left({f, \s_{\bol}}^{-1}\left(c\right)\right)
		\qquad 
		\textnormal{for all }c \in \RR_{\bol},
		\label{eq:condition_for_null}
	\end{align} 
	where the first equality is the {well-known compatibility
		condition for the Neumann problem (see \cite[cor. 9.3]{fabmenmit98} and the remark following the proof)}, and the second equality follows since the single layer potential is symmetric. 
	The solution of \eqref{eq:Neumann-prob} {can then be expressed
		as (cf. \cite[thm. 9.2]{fabmenmit98}):}
	\begin{align}
		\varphi = \s_{\cheese\ot\bcheese} \left(-\frac{1}{2}I+\dlay^{\star}\right)^{-1} {\s_{\bcheese}}^{-1} \left(f\right), \label{eq:harm_grad}
	\end{align}
	which defines the space $\mathscr{I}_{\olt}$.
	This proves {that}
	\[{
		\nulls[	\p_{ \olt\ot\cheese \vpsum}]= \nabla \sobo \cheese \oplus \sobdivo \cheese \oplus \nabla \mathscr{I}_{\olt},}\]
	which is the second equality in \eqref{eq:null-space-general}.
	
	To prove the first {one,}
	i.e. $\nulls[	\p_{\olt\ot\cheese	\vpsum}]= \left(\nabla\harmo {\olt^{\outside}}|_{\cheese} \right)^{\perp}$, we need to characterize
	{ $\nabla \harm \ol \ominus \nabla \mathscr{I}_{\olt}$} (the orthogonal complement of $\nabla \mathscr{I}_{\olt}$ {in $\nabla \harmo \ol$}). Assume that $\nabla h$ is in $\harmo\ol \ominus \mathscr{I}_{\olt}$ and denote its trace by $b = \tr_{\bcheese}(h)$. {Equivalently,}
	\begin{align}
		0	= \left\langle\nabla \varphi, \nabla h\right\rangle_{\ltd \cheese}
		= \left(\pd{\varphi}{\eta}, b \right)
		= \left({\s_{\bcheese}}^{-1}\left(f\right),b\right)
		= \left(f , {\s_{\bcheese}}^{-1}\left(b\right)\right) \label{eq:orthogonality-relation}
	\end{align}
	holds for every $\varphi \in \mathscr{I}_{\olt}$ and thus for each $f\in\sobs{\half}{\bol\setminus\bolm}$ {satisfying} \eqref{eq:condition_for_null}. %
	{It is} evident that \eqref{eq:orthogonality-relation} holds if and only if $b$ is of the form
	\begin{align}
		b=\s_{\bcheese}(k) + c 
		&& 
		\left( \text{for some }k\in\sobs{-\half}{\bolt} \text{ and some } \ c \in \RR_{\bol}\right).
		\label{eq:condition-on-b} 
	\end{align} 
	{As $\RR_{\bol}$ is the space of traces} of functions in $\RR_{\ol}${, we deduce that $h$} is of the form
	\begin{align}
		h = \s_{\cheese\ot\bcheese}\left(k\right) + \chi, && \left(\text{for some }k\in\sobs{-\half}{\bolt} \text{ and some } \ \chi \in \RR_{\ol}\right).
		\label{eq:harmonic-orthogonal}
	\end{align}
	We can {rewrite} the first summand as
	$	\s_{\ol\ot\bol}\left(k\right) 
	= \s_{\ol\ot\bolm}\left(k\right)
	= \left. \s_{\olt^{\outside}\ot\bolm}\left(k\right)\right\vert_{\ol},$
	{because $k$ is supported on $\bolm$} and
	{$\ol$ is} contained in $\olt^{\outside}$. Thus the first term in \eqref{eq:harmonic-orthogonal} is an element of $\harm {\olt^{\outside}} \vert_{\ol}$. After taking the gradient, the second term in \eqref{eq:harmonic-orthogonal} vanishes, which proves that a harmonic gradient in $\nabla \harmo \ol \ominus \nabla \mathscr{I}_{\olt}$ can be harmonically extended to the entire $\olt^{\outside}$.
	The converse statement {holds since every function $\varphi\in\harmo {\olt^{\outside}}$ is of the form $\varphi=c+\s_{\olt^{\outside}\ot\bolm}\left(k\right)$ for some $k\in\sobs{-\half}{\bolm}$ and some constant $c$, by the discussion in Section \ref{DLPS}}. Because the support of $k$ and the support of $f$ in \eqref{eq:orthogonality-relation} are disjoint,
	{$\nabla\varphi$ is trivially orthogonal to $\nabla \mathscr{I}_{\olt}$. Thus,} we obtain the desired condition
	\begin{align}
		\nabla h \in \left.\nabla\harmo{\olt^{\outside}}\right\vert_{\cheese} 
		&&\Longleftrightarrow &&
		\nabla h \in \nabla \harm {\ol} \ominus \nabla \mathscr{I}_{\olt}.
	\end{align}
	This, together with the Helmholtz-Hodge decomposition,
	{completes the proof.}
\end{proof}

{ Theorem \ref{thm:null_spaces} still holds when $\Omega$ is unbounded.  The bounded case, however, is enough to warrant
	most applications.}
\begin{expl}
	In geo-sciences, to model the lithosphere of the Earth, we would choose $\Omega$ to be a Lipschitz cheese with a single cavity (region C in Figure \ref{fig:graincheese}), and $\olt$ to be the exterior of the Earth. This reflects the typical situation where one knows the magnetic potential only outside the Earth (for example on a satellite orbit). According to the previous theorem, {that} part of the lithosphere magnetization that can be recovered uniquely from the magnetic potential { consists of its projection
		onto harmonic gradients} in the lithosphere that can be continued {harmonically throughout} the entire Earth. {This projection
		is in fact the magnetization with minimum $L^2(\Omega;\mathbb{R}^d)$-norm
		producing the observed field.}
\end{expl}
{For $\Omega$ a bounded Lipschitz open set,
	both $\nabla \sobo \cheese$ and $\sobdivo \cheese$ consist of integrable vector fields with zero mean on $\Omega$. Therefore, we immediately deduce from
	Theorem \ref{thm:null_spaces} the following fact:}
\begin{cor}\label{cor:mean}
	{Let $\cheese$ be a bounded Lipschitz open set, and  $\funf\in L^2(\Omega;\mathbb{R}^d)$ a vector field which is invisible
		in $\Omega^{\outside}$. Then $\int_\Omega\funf(y)d\mu(y)=0$.}
\end{cor}
{We note that Corollary \ref{cor:mean} also follows from \cite[Lem. 2.8]{barvilhar19}.} The upcoming corollary connects our {work} to the main result in \cite{fabian19}{,} as well as {to} \cite[Lem. 2.4]{barvilhar19}. {Roughly speaking,} it says that a {finite} collection of  magnetized grains {that do not touch} is invisible if and only if each grain {is itself} already invisible. 

\begin{cor}\label{cor:cluster}
	Let $\cheese$ be a Lipschitz cluster in $\Rd$ ($d\geq 3$) with $N$ connected components $\os_{1},\ldots,\omega_N$. Then 
	\begin{align}\label{corcc}
		\nulls	\left[	\p_{	\cheese^{\outside} \ot \cheese	\vpos	}	\right] 
		= \bigoplus_{i=1}^{N} \left(	\nabla \sobo{\os_{i}} + \sobdivo {\os_{i}}	\vpsum	\right),
	\end{align}
\end{cor}
\begin{proof}
	{Observe that $\mathscr{I}_{\ol^{\outside}}$ is trivial.
		Indeed, from \eqref{eq:form-of-null-harmonics},} we know that $\mathscr{I}_{\ol^{\outside}}$ is parametrized by functions $f\in\sobs{\half}{\bol\setminus\bol^{\outside}}${,} and since $\bol^{\outside} = \bol$ it follows that $\mathscr{I}_{\ol^{\outside}}=\{ 0 \}$. {Thus, by Theorem \ref{thm:null_spaces}, we need only show that $\nabla \sobo \cheese \oplus \sobdivo \cheese$ splits further into the direct sum indicated by \eqref{corcc}.
		This is} immediate, since $\ol$ is a union of {positively separated Lipschitz open} sets.
\end{proof} 
{It is worth pointing out that, as soon as the connected components of an open set $\Omega$ are at positive 	distance from each other, a field supported on $\Omega$ is invisible if and only if its restriction to each component is invisible, no matter if $\Omega$ is 	Lipschitz or not \cite[Lem. 2.4]{barvilhar19}. 	Note also that the result no longer holds if the connected components do touch, even when each component is Lipschitz,  cf.  Figure \ref{fig:divfree}.}

In geophysics, {magnetizations are often} considered to be induced by some strong ambient magnetic field. If the sample under consideration is small, then {one can assume that the magnetization is constant in direction throughout the sample}, and proportional {in magnitude} to the material's susceptibility{; we call such a sample {\em unidimensional}}. {And as  rock materials mingle during subsequent history, it seems fairly reasonable to assume that a rock sample 	consists of a concatenation of unidimensional samples.} The next corollary shows that, if the magnetization in each grain of a cluster {is unidimensional}, then the directions in each are determined uniquely or the susceptibilities have zero mean (i.e., the magnetization in each grain has a vanishing net moment).  

\begin{cor}\label{cor:unidirect}
	Let $\ol$ be a Lipschitz cluster in $\Rd$ ($d\geq 3$) with $N$ connected components $\os_{1},\ldots,\os_{N}$. For $i\in\{1,\dots,N\}$ let $\vect_i,\mathbf{w}_i$ denote unit vectors in $\Rd$ and $f_i,g_i$ be functions in  $L^2(\os_i)$. Further, let $\funf$ and $\fung$ be vector fields of the form
	\begin{align}
		\funf=\sum_{i=1}^N\vect_i f_i\ \chi_{\omega_i},
		&&
		\fung=\sum_{i=1}^N\mathbf{w}_i g_i\ \chi_{\omega_i}.
		\label{eq:unidirectional}
	\end{align}
	If $\funf$ and $\fung$ generate the same magnetic potential, i.e., $\p_{\cheese^{\outside}\ot\cheese}(\mathbf{f})=\p_{\cheese^{\outside}\ot\cheese}(\mathbf{g})$, then on each $\os_{i}$  it holds $\left\langle  f_{i}, 1 \right\rangle_{\lt {\os_{i}}} =  \left\langle  g_{i} , 1 \right\rangle_{\lt {\os_{i}}}$ and
	\begin{align}
		\text{either} \quad \left\langle  f_{i}, 1 \right\rangle_{\lt {\os_{i}}} = 0
		&& \text{or} \quad \vect_{i} = \mathbf{w}_{i}.
	\end{align} 
\end{cor}

\begin{proof}
	The assumption $\p_{\cheese^{\outside}\ot\cheese}(\mathbf{f})=\p_{\cheese^{\outside}\ot\cheese}(\mathbf{g})$ implies $\funf-\fung\in \nulls(\p_{\cheese^{\outside}\ot\cheese})$. {Thus, by Corollary \ref{cor:cluster},
		$(\vect_i f_i-\-\mathbf{w}_i g_i)\ \chi_{\omega_i}$ is invisible in $\Omega^{\outside}$  and therefore, by Corollary \ref{cor:mean},
		$\vect_{i} \left\langle f_{i} , 1 \right\rangle_{\lt {\os_{i}}} = \mathbf{w}_{i} \left\langle g_{i} , 1 \right\rangle_{\lt {\os_{i}}}$. From this, the result follows at once.}
\end{proof}

{Still} within the context of geo-science{s}, Theorem \ref{thm:null_spaces} implies that {one} cannot recover the full magnetization of the lithosphere knowing only its magnetic potential
{in the exterior of}  the Earth. But it leaves open {the question whether or not one can} estimate the depth of the magnetized lithosphere to some extent from the magnetic data. The next corollary shows that  even this is {not} possible without further assumptions. 

\begin{cor}\label{cor:suppnull}
	Let $\ol$ be a bounded, connected and simply connected Lipschitz domain, and
	{$\olm$} a Lipschitz cluster contained in $\ol${;} i.e. $\overline{\olm}\subset\ol$. Then{,} for every vector field $\funf\in\ltd \ol${,} there exists a vector field $\fung\in
	L^2(\Omega\setminus{\overline{\olm}};\mathbb{R}^d)$ such that
	\begin{align*}
		\p_{\ol^{\outside}\ot\ol}\left(\funf\right) 
		= \p_{\ol^{\outside}\ot\ol\setminus{\overline{\olm}}}\left(\fung\right).
	\end{align*}
\end{cor}

\begin{proof}
	If $P_{\ol^{\outside}\ot\ol}(\funf)=0$ then $\fung=0$ satisfies the condition. Hence, by Theorem \ref{thm:null_spaces} it suffices to consider vector fields of the form $\funf=\nabla \varphi$ with $\varphi\in\harm \ol$.
	
	Let $\Gamma$ be a Lipschitz cluster such that each connected component of $\Gamma$ is an open neighborhood of exactly one connected component of $\olm$. Such a cluster exists because the connected {components of the Lipschitz open set $\olm$ are positively separated}.
	Let $\eta$ denote a smooth cut-off function {equal to $1$} on $\overline{\olm}$ and {to $0$} on $\Gamma^{\outside}$. The function $\eta\varphi$ is then in $\sobo \ol$, and by Theorem \ref{thm:null_spaces} its gradient is invisible on $\ol^{\outside}$. Moreover, the support of $\nabla (\varphi - \eta\varphi)$ is contained within $\ol\setminus\olm$ which implies{, since $\partial\olm$ has $\mu$-measure zero, that}
	\begin{align*}
		\p_{\ol^{\outside}\ot\ol} \left(	\nabla \varphi	\right) 
		= \p_{\ol^{\outside}\ot\ol} \left(	\nabla (\varphi - \eta\varphi) 	\right)
		= \p_{\ol^{\outside}\ot\ol\setminus\olm} \left(	\nabla (\varphi - \eta\varphi)	\right)\\=\p_{\ol^{\outside}\ot\ol\setminus{\overline{\olm}}} \left(	\nabla (\varphi - \eta\varphi)	\right).
	\end{align*}
	Setting $\fung = \nabla (\varphi - \eta\varphi)$ completes the proof.
\end{proof}

\subsection{Stability estimate for harmonic gradients}\label{sec:stab} 

In geophysical applications{, one measures} the magnetic potential {in a limited region only}, say on a satellite orbit $\bol$ around the Earth {$\omega$, with $\omega \subset \ol$}. To reconstruct the magnetization of the lithosphere from {such} data, we need to continue the measured potential to the entire region $\os^{\outside}$ and then invert for the magnetization. It is well-known that this problem is ill-conditioned, {the main source of instability being} the harmonic continuation ({as indicated for instance} in \cite{isakov12}). Assuming that we already know the magnetic potential everywhere in $\os^{\outside}$, the inversion for the harmonic gradient within $\os$ {that produces this potential} is stable, as the following proposition shows. In fact, we actually consider stability with respect to the gradient of the magnetic potential, i.e., with respect to the magnetic field.
\begin{prop}\label{lem:stability}
	Let $\cheese$ be a bounded, connected and simply connected Lipschitz domain in $\Rd$ ($d\geq3$). There is a constant $C>0$ such that, for every $\mathbf{f}\in\nabla\harmo {\cheese}$,
	\begin{align}\label{inegstabc}
		\left\|\mathbf{f}\right\|_{L^2(\cheese,\Rd)}\leq C\left\|\nabla\p_{ \cheese^{\outside}\ot\cheese\vpsum}(\mathbf{f})\right\|_{{L^2(\cheese^{\outside};\mathbb{R}^d)}}.
	\end{align}	
	\label{eq:harmonic-estimate}
\end{prop}

\begin{proof}
	In {what follows,} $C$ will denote a positive
	{constant that depends on $\Omega$ and whose value may be different each time it
		occurs. When $\varphi\in\harmo \ol$, it is a classical estimate for the Neumann problem on a bounded domain \cite[thm. 9.2]{fabmenmit98} that}
	\begin{align}
		\| \nabla \varphi \|_{\ltd \ol}
		\leq
		C \left\| \pd{\varphi}{\nu} \right\|_{\sobs{-\half}{\bol}}.
		\label{eq:third-estimate}
	\end{align}
	{Also, by partial integration, we get with $\funf = \nabla \varphi$ that
		$\p_{ \ol^{\outside}\ot\ol\vpsum}(\mathbf{f})=\s_{\cheese^{\outside}\ot\bcheese}\left(\pd{\varphi}{\nu}\right)$. Let $R>0$ be large enough that $\mathbb{B}_R(0)\supset\overline{\Omega}$, and put $\Sigma:=\mathbb{B}_R(0)\cap\Omega^c$ (a shell-shaped domain).
		Clearly, 
		$|\p_{ \ol^{\outside}\ot\ol\vpsum}(\mathbf{f})(x)|=O(1/|x|^2)$ for $|x|$ large,
		hence $\s_{\cheese^{\outside}\ot\bcheese}\left(\pd{\varphi}{\nu}\right)
		\in \harmo {\ol^{\outside}}$ because its gradient is square summable in a neighborhood of $\infty$, as well as  in $\Sigma$ by an estimate analogous to \eqref{eq:third-estimate}. Adding a constant to $\varphi$ if necessary, we may assume it has zero mean on $\Sigma$ as this does not affect $\nabla\varphi$. In view of the continuity of the trace operator in the bounded Lipschitz domain $\Sigma$ and  the Poincar\'e inequality, we get a fortiori} that 
	\begin{align*}
		\left\|\nabla\s_{\cheese^{\outside}\ot\bcheese}\left(\pd{\varphi}{\nu}\right)
		\right\|_{L^2(\Omega^c;\mathbb{R}^d)} 
		 \geq
		C \left\|
		\s_{\bol}\left(\pd{\varphi}{\nu}\right)
		\right\|_{H^{\frac{1}{2}}(\bol)}
		\mkern-14mu.
	\end{align*}
	{Altogether, since $S_{\partial\Omega}:H^{-1/2}(\partial\Omega)\to S^{1/2}(\partial\Omega)$ is invertible, we obtain:}
	\begin{align*}
		&\left\|{\nabla}
		\p_{ \ol^{\outside}\ot\ol\vpsum}(\mathbf{f})
		\right\|_{{L^2(\Omega^c;\mathbb{R}^d)}}=
		\left\|{\nabla}
		\s_{\cheese^{\outside}\ot\bcheese}\left(\pd{\varphi}{\nu}\right)
		\right\|_{{L^2(\Omega^c;\mathbb{R}^d)}} 
		\\&\geq
		C \left\|
		\s_{\bol}\left(\pd{\varphi}{\nu}\right)
		\right\|_{H^{\frac{1}{2}}(\bol)}
		\geq 
		C \left\|
		\pd{\varphi}{\nu}
		\right\|_{H^{-\half}(\bol)}.\label{eqn:lemest1}
	\end{align*}
	Combining this with \eqref{eq:third-estimate} concludes the proof.
\end{proof}

\begin{rem}
	{Of course, Proposition \ref{lem:stability} still holds (with
		a different constant but essentially the same proof) if we replace the $L^2(\Omega^c;\mathbb{R}^d)$-norm
		in the right hand side of \eqref{inegstabc} by the $L^2(A;\mathbb{R}^d)$-norm, where $A$ is a Lipschitz shell of the form $O\cap\Omega^c$ with $O$ a Lipschitz open set containing $\overline{\Omega}$. However}, we want to point out that the above proposition implicitly requires harmonic continuation. To see this, {consider} a harmonic gradient $\nabla \varphi$ on a small Lipschitz domain $\os$ that is strictly embedded in $\ol$. Knowing its magnetic potential on $\ol^{\outside}$, the proposition assures only that we can reliably reconstruct the harmonic part of the field $\nabla \varphi\wedge 0$ (the original field {extended by zero to the whole of} $\ol$), but not the field $\nabla \varphi$ itself. In fact, since $\os$ and $\ol^{\outside}$ are {positively separated,} the operator $\p_{\ol^{\outside}\ot\os}$ is compact and its spectrum {contains zero}. Hence, no stability result as in proposition \ref{lem:stability} can hold {on} an infinite dimensional subspace of $\ltd \os$; in particular, not {on} $\nabla \harmo \os$.  
\end{rem}

\section{Simple vector fields}\label{sec:unique}

In the previous section{, we showed} one can reconstruct the {``harmonic gradient part'' of a} magnetization, given its magnetic potential. In this section we define a different {class}
of vector fields {enjoying} similar uniqueness properties{, but having} a better geometrical interpretation.

\paragraph{Simple vector fields.}
We call a vector field \textit{simple} if {it is a simple
	function on $\mathbb{R}^d$. That is, $\funf$ is simple if its essential
	range is finite:}
\begin{equation}
	\label{decs}
	\funf = \sum_{i=1}^{N} \vect_{i} \ \chi_{\os_{i}}
\end{equation}
{ for some finite family of Borel sets $\os_{1},\dots,\os_{N}$ and of vectors $\vect_{1},\dots,\vect_{N}$ in $\mathbb{R}^d$.} We refer to the family $\os_{1},\dots,\os_{N}$ as to a \textit{representation} of $\funf$. We call $\funf$ \textit{box-simple} if it is simple and admits a representation that contains only {bounded rectangular parallelepipeds} {with edges}  parallel to the {axes} of the coordinate system. 
\vspace{0.25 cm}

Notice that {a} representation of a simple vector field is non-unique{: not only can we change $\omega_i$ by a set of measure
	zero, but we can also} partition the support of a simple field into {smaller} sets and appropriately adjust the family of vectors. Nevertheless, the concept of a representation {will be useful in what follows}.
{A nonzero simple field can be  invisible,}  as the following two examples show.

\begin{expl}\label{exp:triangle}
	Let $T_{1}$ denote a triangle in $\RR^{2}$ defined by the vertices at $(0,0), (0,1)$, and $(\half,\half)$. Let $T_{2}, T_{3}$, and $T_{4}$ denote the subsequent rotations of $T_{1}$ by $\pi /2$ around the point $(\half,\half)$. Let $P_{i}=T_{i} \times [0,1]$ denote  the corresponding prisms in $\RR^{3}$ and write $\{ e_{1},e_{2},e_{3}\}$ for the standard basis vectors. Consider the simple vector field 
	\begin{align}
		\funf = e_{1} \ \chi_{P_{1}} + e_{2} \ \chi_{P_{2}} - e_{1} \ \chi_{P_{3}} - e_{2} \ \chi_{P_{4}}.
		\label{eq:inv_tri}
	\end{align}
	Since this vector field is divergence-free and has a vanishing normal almost everywhere on the surface of the cube $Q = P_{1}\cup P_{2}\cup P_{3} \cup P_{4}$, it is invisible. An illustration is provided in Figure \ref{fig:divfree}.
\end{expl}

\begin{figure}[t]
	\centering
	\begin{subfigure}{0.35\textwidth}
		\centering
		\def\svgwidth{4.5cm}
		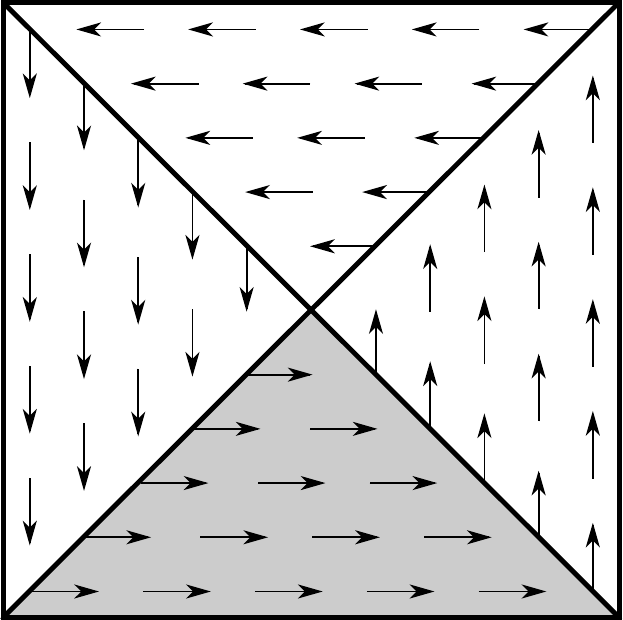
		\subcaption{Outline of the vector field in \eqref{eq:inv_tri}}
		\label{pic:a}
	\end{subfigure}\hspace{2cm}
	\begin{subfigure}{0.35\textwidth}
		\centering
		\def\svgwidth{4.5cm}
		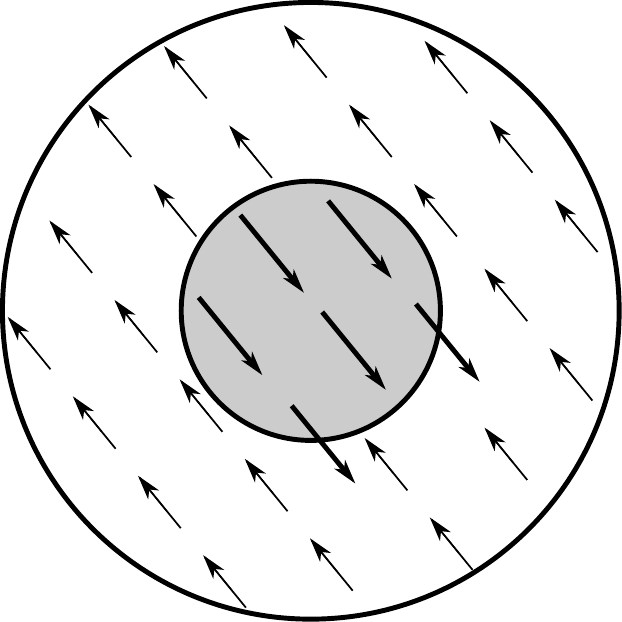
		\subcaption{Outline of the vector field in \eqref{eq:inv_ball}.}
		\label{pic:b}
	\end{subfigure}\hspace{0.6cm}
	\caption{Invisible simple-vector fields: The figure \ref{pic:a} shows the profile of the example \ref{exp:triangle}. The triangle $T_{1}$ is depicted gray. The figure \ref{pic:b} illustrates the example \ref{expl_ball}. The magnetization of the smaller ball (in gray) is by the factor $\alpha^{-1}$ larger than that of the larger ball.}
	\label{fig:divfree}
\end{figure}

\begin{expl}\label{expl_ball}
	For $0<\alpha<1$, $\vect\in\Rd$ and $\BB_r(0)\subset\Rd$,  a ball of radius $r>0$ centered at the origin, the field
	\begin{align}
		\funf = \vect \  \chi_{\BB_r(0)} - \frac{\vect}{\alpha^{d}} \chi_{\BB_{\alpha r}(0)}
		\label{eq:inv_ball}
	\end{align}
	is simple. It is invisible since
	\begin{align}
		\p_{\BB_r(0)^{\outside}\ot\BB_r(0)}\left( \vect \right) (x) 
		=  {\ \frac{ r^{d}}{(d-2)\| x \|^{d}}  \ \vect \cdot x}
		= \p_{\BB_r(0)^{\outside}\ot\BB_{\alpha r}(0)}\left( \frac{\vect}{\alpha^{d}} \right) (x).\label{eqn:dipole}
	\end{align}
	{This last example suggests} that the magnetic potential of simple fields {representable} on balls contains no information about the size of these balls. However, if we know a priori that the centers of the balls are pairwise disjoint,
	{then at least the locations of the centers} are determined uniquely. Since $\frac{ r^{d}}{(d-2)\| x \|^{d}}  \ \vect \cdot x$ is the field of a dipole with moment ${\bf v}\frac{r^d}{d-2}$ located at the origin, the latter is {equivalent} to the known statement that a collection of finitely many disjoint dipoles is determined uniquely by {its} magnetic potential.
\end{expl}

Simple fields suit the setup described, e.g., in \cite{degroot18}, where grains $\omega_i$ within a rock sample are identified via computed tomography (microCT), and subsequently magnetic field data from scanning superconducting quantum interference microscopy (SSM) measurements are inverted for a constant magnetization within each grain. {Even though the above examples show that nonzero} simple fields can be invisible,
in the next section we prove that this cannot happen for box-simple vector fields. {Therefore,} their magnetic potential uniquely determines their magnetization{, and a priori} localization of the grains via micoCT would thus (in theory) not be necessary for such fields.

\subsection{Uniqueness for box-simple vector fields}
The above examples show that {nonzero simple vector fields
	can be pure} elements of $\nabla\sobo \ol \oplus \sobdivo \ol$.
Quite surprisingly, though, this cannot happen for box-simple vector fields.
{Indeed}, the next theorem states that a {nonzero} box-simple field is always visible. {When $\funf$ is a nonzero box-simple field given by \eqref{decs}, we tacitly assume that the $\omega_i$ are nonempty open rectangular parallelepipeds and the $\vect_i$ are all nonzero, so that $\supp{\funf}=\bigcup_i\overline{\omega}_i$ .}

\begin{thm}\label{thm:simple-funtions}
	{If $\funf$ is} a nonzero box-simple vector field{,  then}
	\begin{align}
		\p_{\supp{\funf}^{\outside}\ot\supp{\funf}} \left(	\funf	\right) \neq 0.
	\end{align}
	If $\ol$ is a bounded, connected, and simply connected Lipschitz domain that contains $\supp{\funf}$, then the Helmholtz-Hodge decomposition of $\funf$ {in $\ol$} contains a non-vanishing harmonic gradient.
\end{thm}

\begin{rem}\label{rem:unitary-trafo}
	Before we prove this statement{, let us stress that rigid motions} preserve invisibility. More precisely, let $U\colon \Rd \to \Rd$ be {such an affine transformation}. If $A$ is a set in $\Rd$, we define $U(A)= \left\{	U x  : x \in A	\right\}$ and notice that $U(A^{\outside})=U(A)^{\outside}$. {Then,} $U$ induces a linear transformation $\mathbb{U}\colon \Rd \to \Rd$ on vectors in $\Rd$ such that a vector $\vect \coloneqq x-y \in \Rd $ is mapped to $ \mathbb{U} \vect \coloneqq Ux -Uy$. This $\mathbb{U}$ is orthogonal with respect to the Euclidean scalar product{,} and we denote its adjoint by {$\mathbb{U}^{\star}=\mathbb{U}^{-1}$}.
	
	{Now, if $\funf$ is} invisible on $\ol^{\outside}$, then the field $\funf_{U}$ defined as $\funf_{U}(x) = \mathbb{U}^{\star} (\funf (U x))$ { is invisible on $U^{-1}(\ol^{\outside})$,} since for every $x\in U^{-1}(\ol^{\outside})$ we have {that}
	\begin{align*}
		0 	= \p_{\ol^{\outside}\ot \ol} \left(	\funf	\right) (U x) 
		&= \frac{1}{C_{d}} \int_{\ol} \frac{Ux - y}{|Ux - y|^{d}} \cdot \funf(y) \measure{y} \\
		&= \frac{1}{C_{d}} \int_{U^{-1}(\ol)} \frac{x - y}{|x - y|^{d}} \cdot \funf_{U}(y) \measure{y} 
		= \p_{U^{-1}(\ol^{\outside})\ot U^{-1}(\ol)}(\funf_{U}) (x).
	\end{align*}
	Here, for the third equality we used the change of coordinates $y\mapsto Uy$ and the fact that $|\det(U)| =1$.
\end{rem}

\begin{proof}[Proof of Theorem \ref{thm:simple-funtions}]
	We will prove the statement by contradiction, assuming that $\funf$ is nonzero and invisible. %
	
	Let $\os_{1},\dots , \os_{N}$ be a box representation of $\funf$. Since intersections of boxes are again boxes, we can assume without loss of generality that $\os_{i}$'s are pairwise disjoint (but their boundaries {can touch)}. {Write $O:=\bigcup_j\os_j$ so that $\supp{\funf}=\overline{O}$,
		and put $r = \sup \{ e_{1} \cdot x \ : \ x \in O \}$} for the {largest abscissa in the $e_{1}$-direction within $\overline{O}$.}
	Let $P$ be {the} hyper-plane that is perpendicular to $e_{1}$ {and crosses} the $e_{1}$-axis at the point $r$. {This
		plane} intersects {$\overline{O}$}
	{along one} face  of finitely many boxes
	{$\overline{\os}_{i_{1}},\dots ,\overline{\os}_{i_{M}} $,
		and  there exists a number $l>0$ so small that} the shifted plane $P_{2l} \coloneqq P - 2l e_{1} $ intersects {$O$}
	{in} the same family of boxes{:} $P_{2l}\cap
	{O} = P_{2l} \cap (\os_{i_{1}}\cup\dots \cup\os_{i_{M}} )$. 
	
	Now, consider {the} affine translation $T(x) \coloneqq x + l e_{1}$, and put  for brevity $K \coloneqq {O} \cup T^{-1}({O})$. For $a,b \in \RR$  with $b\geq a$,  we define a slice in {the} $e_1$-direction by $\sli{a,b} = \{ x \in \Rd  \ : \ e_{1}\cdot x \in (a,b] \}$. Then, the domain $K$ splits into three {slices:}
	\begin{align*}
		K = 
		\underbrace{K \cap \sli{-\infty,r-2l}}_{C}
		\cup
		\underbrace{ K \cap \sli{r-2l, r-l}}_{B}
		\cup 
		\underbrace{K \cap \sli{r-l, r}}_{A}.
	\end{align*}
	
	A {corresponding  picture} in 2D is {shown} in Figure \ref{fig:2d-construction} {to illustrate the construction}. The {translation} $T$ induces {on vectors}
	a unitary operator $\mathbb{T}$ %
	{which is simply the identity}.
\begin{figure}[t]
		\centering
		\def\svgwidth{10cm}
		\input{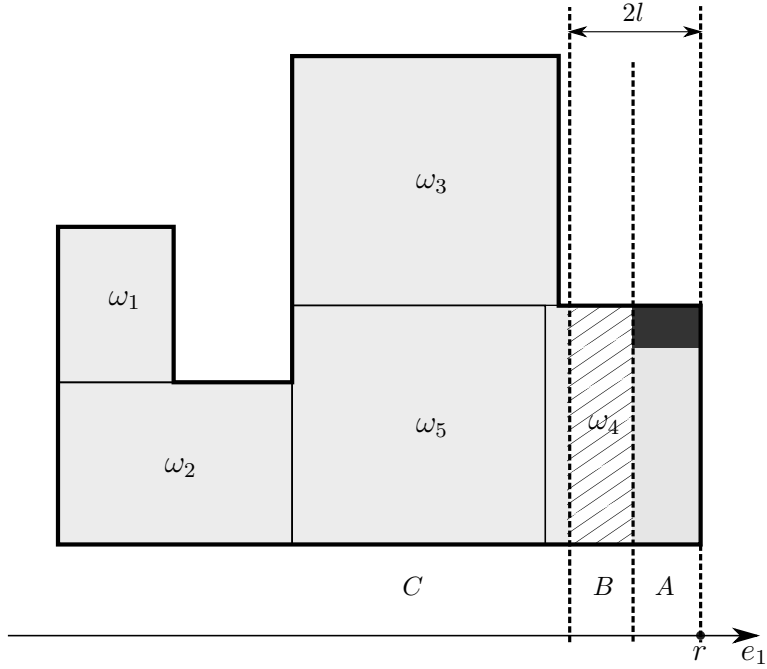}
		\caption{Constructed regions $A$, $B$, and $C$, for {$O$} being a union of five cubes $\omega_1,\dots,\omega_5$. The dashed lines indicate the right boundaries of the slices $\sli{-\infty, r-2l}$, $\sli{r-2l, r-l}$, and $\sli{r-l, r}$. The small dark box within $\os_{4}$ shows a possible support of the function $\funf_{d}$ that arises after subsequently slicing $\funf$ in the directions $e_{1}$ to $e_{d}$ (here, $d=2$).}
		\label{fig:2d-construction}
	\end{figure}
	The {corresponding} shifted vector field {is
		therefore} $\funf_{T} = \mathbb{T}^{\star} \funf \circ T = \funf \circ T$. From  Remark \ref{rem:unitary-trafo}  {it follows} that $\funf$ and $\funf_{T}$ are both invisible on $K^{\outside}$. {Thus,} we have
	{that}
	\begin{align}
		0	
		\!= 	\p_{K^{\outside}\ot K} \left(\funf - \funf_{T} \right)
		\!= 	\p_{K^{\outside}\ot C}\left( \funf - \funf_{T} \right)
		+ \p_{K^{\outside}\ot B}\left( \funf - \funf_{T} \right)
		+ \p_{K^{\outside}\ot A}\left( \funf - \funf_{T} \right).
		\label{eq:split-of-potential}
	\end{align}
	Consider the last term on the rhs. of \eqref{eq:split-of-potential}:
	for $x\in A$, we have $Tx \notin {O}$ {by the maximality of $r$, and consequently}
	\begin{align*}
		\funf_{T}(x)=\funf (Tx) = \sum_{i=1}^{N} \chi_{{\omega_{i}}}(Tx) \ \vect_{i} = 0.
	\end{align*}
	
	Next, consider the middle term on the rhs. of \eqref{eq:split-of-potential}: for $x\in B$ {we have $Tx\in A$,} and since by construction both $B$ and $A$ intersect the same boxes in {$O$}, we have
	\begin{align*}
		\funf_{T}(x) = \funf (Tx) = \sum_{i=1}^{N} \chi_{{\omega_{i}}}(Tx) \ \vect_{i} 
		=  \sum_{i=1}^{N} \chi_{{\omega_{i}}}(x) \ \vect_{i} 
		= \funf (x).
	\end{align*}
	Therefore, this {middle} term vanishes and {equation \eqref{eq:split-of-potential} can be rewritten as}
	\begin{align}
		0	=  
		\p_{K^{\outside}\ot C}\left( \funf - \funf_{T} \right)
		+ \p_{K^{\outside}\ot A}\left( \funf \right).
		\label{eq:split-of-potential-simplified}
	\end{align}
	The first term in \eqref{eq:split-of-potential-simplified} defines a harmonic function on $K^{\outside}$ that {extends on the larger region $C^{\outside}$ to a harmonic function $h_1$. Likewise, the second term
		extends  harmonically  to $h_2$ defined on $A^{\outside}$.
		The domain $D:=\{x:\,r-2l<x_1<r-l\}$ is included in $A^{\outside}\cap C^{\outside}$, therefore $h_1+h_2$ is harmonic there. For $R>0$ large enough
		$D\setminus \mathbb{B}_R(0)\subset K^\outside$, and since $h_1+h_2=0$ in $K^\outside$
		by \eqref{eq:split-of-potential-simplified} we deduce that
		$h_1+h_2$ is identically zero in $D$. Hence, the function $h$ which is
		$h_1$ on $\{x;\,x_1>r-2l\}$ and $-h_2$ on $\{x:\,x_1<r-l\}$ is well-defined and harmonic,} 
	everywhere on $\Rd${, moreover} it vanishes at infinity. By Liouville's theorem we thus have $h=0${, in particular
		$h_2=0$ in $\{x:\,x_1<r-l\}$ and consequently}
	\begin{align*}
		\p_{A^{\outside}\ot A}(\funf) = 0,
	\end{align*}
	{since $\overline{A}$ consists of finitely many disjoint
		closed polytops $P_j$, 
		and therefore  $A^{\outside}$ is connected.}
	Hence, the restriction of $\funf$ to the slice $A$ is invisible{, and by Corollary \ref{cor:cluster} so is the restriction of $\funf$ to
		the interior $O_j$ of $P_j$ for each $j$. Note that all the vertices of $O_j$ in the plane $\{x:x_1=r\}$ are vertices of $\supp{\funf}$.}
	{We can now define} $\funf_{1} = \funf\vert_{O_1}$ and repeat the
	{previous} construction with $\funf_{1}$ instead of $\funf$, only with $e_2$ in place of $e_1$. This {yields that the restriction of $\funf_{1}$ to a small slice $O_2$ of $O_1$ in the direction $e_{2}$ is invisible} (compare to the dark region in Figure \ref{fig:2d-construction}). {Note that every vertex of $O_2$ with maximal coordinates $x_1$, $x_2$ is a vertex of $\supp{\funf}$.  We then define $\funf_2:={\funf_1}|_{O_2}$
		and repeat the construction for} $\funf_{2}$ in the direction $e_{3}$. Continuing this process{,} 
	we eventually get a field $\funf_{d}$ { which is invisible, by construction}. {Moreover}, $\funf_{d}$ is {the} restriction of the original field $\funf$ {to} a small box with edge {lengths} $l_{1},\dots,l_{d}$ that shares a vertex with one of the outer corners of {$O$, therefore}  this box is contained entirely in one of the boxes  $\os_{1},\dots, \os_{N}$. Consequently, $\funf_{d}$ is a non-vanishing constant vector field, which is a harmonic gradient, defined on a Lipschitz domain. Theorem \ref{thm:null_spaces}, however, states that such a vector field cannot be invisible. This yields a contradiction and proves the first assertion. 
	
	For the second assertion observe that if $\ol$ contains $\supp{\funf}$ then $\supp{\funf}^{\outside}$ contains $\ol^{\outside}$. From the first part of the theorem we know that $\funf$ is visible on $\supp{\funf}^{\outside}$ and thus it is in particular visible on $\ol^{\outside}$. From Theorem \ref{thm:null_spaces} it follows that the Helmholtz-Hodge-decomposition of $\funf$, extended by zero to $\ol$, contains a {nonzero} harmonic gradient.
\end{proof}

Theorem \ref{thm:simple-funtions} can be generalized to simple fields with more general representations than rectangular parallelepipedic boxes, but we shall not pursue this here. {An interesting feature of box-simple fields is that they form a vector space, because rectangular parallelepipedic boxes
  are stable under intersections. Hence, Theorem \ref{thm:simple-funtions} implies
  the following uniqueness result.}
\begin{cor}\label{cor:uniquerect}
	Let $\funf$ and $\fung$ be box-simple vector fields with supports contained in some open set $\Omega\subset\Rd$ and
	\begin{align*}
		\p_{\cheese^{\outside}\ot\cheese}(\mathbf{f})=\p_{\cheese^{\outside}\ot\cheese}(\mathbf{g}).
	\end{align*}
	Then it holds {that} $\mathbf{f}=\mathbf{g}$.
\end{cor}
\begin{proof}
	The space of box-simple vector fields is a vector space. Therefore, $\funf-\fung$ is box-simple and by assumption $\p_{\cheese^{\outside}\ot\cheese}(\funf-\fung)=0$. By Theorem \ref{thm:simple-funtions}{, it follows} that $\mathbf{f}=\mathbf{g}$.
\end{proof}

\subsection{Stability estimate for box-simple vector fields}
\label{subsec:sb}
As {mentioned already} in the previous section, 
{$\p_{\ol^{\outside}\ot\os}:L^2(\omega;\mathbb{R}^d)\to\sob{\ol^c}$ is a compact operator when $\overline{\os}\subset\ol$}. 
It is clear that box-simple vector fields are dense in $L^{2}(\omega,\Rd)${,}
and thus no stability result can hold for such fields, just as it is the case  for harmonic gradients supported on a smaller Lipschitz domain. {Unlike harmonic gradients}, however, box-simple vector fields {naturally give rise to a scale of finite dimensional vector spaces, for which stability estimates trivially hold. As} opposed to the case of harmonic gradients, {such estimates hold} for vector fields whose support is not all of $\Omega$, and thus {do} not require harmonic continuation. 

\paragraph{Box-simple lattice fields.}
Let $\Lambda_{\delta}$ define a cubic lattice on $\Rd$ with lattice spacing $\delta>0$.
A \textit{box-simple $\Lambda_{\delta}$ field} is a simple vector field that admits a representation whose elements are the cells of $\Lambda_{\delta}$. For an open domain $\ol$ we denote the space of box-simple $\Lambda_{\delta}$ fields with support in $\ol$ by $\RR_{\Lambda_{\delta},\ol}$. Clearly, if $\ol$ is bounded then $\RR_{\Lambda_{\delta},\ol}$ is finite dimensional. 

{Discretization via box-simple lattice fields {is} frequently used in numerical applications{, where they typically combine with  additional} means of regularization depending on the particular goal {one is pursuing} (e.g., sparsity or smoothness). That discretization alone is{, in spite of Proposition \ref{prop:stab},} unlikely to be sufficient {for} adequate regularization of inverse magnetization problems is {illustrated by example \ref{expl:cdelta} below, which indicates} that the bounding constant grows exponentially with the refinement of the discretization scale.}

\begin{prop}\label{prop:stab}
	Let $\ol$ be a bounded, connected and simply connected Lipschitz domain in $\Rd$ ($d\geq3$). There is a constant $C(\delta)>0$ such that, for every $\funf\in\RR_{\Lambda_{\delta},\ol}$,
	\begin{align}
		\| \funf \|_{\ltd \ol} \leq C(\delta)  \ 
		\left\| \tr_{\bol}\left( \p_{{\ol^c}\ot\ol \vpsum} \left( \funf \right) \right)\right\|_{\lt {\bol}}.
	\end{align}
	The constant $C(\delta)$ becomes unbounded as $\delta$ approaches zero.
\end{prop}
\begin{proof}
	By Theorem \ref{thm:simple-funtions} the operator $\p_{\ol^{\outside}\ot\ol}$ 
	is injective on the space of box-simple vector fields. By uniqueness of
	{a solution to the harmonic Dirichlet problem in $\sobloc{1}{\Omega}$
		with data in $H^{1/2}(\partial\Omega)$ (see discussion in the proof of
		Theorem \ref{thm:null_spaces}),} it thus follows that $\p_{\bol\ot\ol}$ is injective as well (here, $\p_{\bol\ot\ol}$ denotes the trace of $\p_{\ol^{\outside}\ot\ol}$ {on} $\partial\Omega$);  restricting $\p_{\bol\ot\ol}$ to $\RR_{\Lambda_{\delta},\ol}$ defines a bijective linear operator on a finite dimensional space. This operator is thus invertible{,} and it follows {when $\funf\in \RR_{\Lambda_{\delta},\ol}$ that}
	\begin{align}
		\| \funf \|_{\ltd \ol} = 
		\left\|	\p_{\bol\ot\ol}^{-1} \left(	\p_{\bol\ot\ol\vpsum}	\left(	\funf	\right)	\right)	\right\|_{\ltd \bol}	
		\leq C \  \left\| \p_{\bol\ot\ol\vpsum} \left(\funf\right)\right\|_{\ltd \bol},
		\label{eq:simple-estimate}
	\end{align}
	for some constant $C>0$.
	{As $\bigcup_{\delta>0}\RR_{\Lambda_{\delta},\ol}$ is dense in $\ltd \ol$,  we can construct} a sequence of fields $\funf_{n}\in\RR_{\Lambda_{1/n},\ol}$ that converges to an element {of} $\nabla\sobo{\ol}$, which is in the nullspace of $\p_{\bol\ot\ol\vpsum}$. Thus{,} the magnetic potential of $\funf_n$ converges to zero, but the left hand side of \eqref{eq:simple-estimate} remains finite, implying the constant $C(\delta)$ has to grow indefinitely {when $\delta\to0$}. 
\end{proof}

\begin{expl}
	\label{expl:cdelta}
	\begin{figure}
		\centering
		\includegraphics[width=0.75\textwidth]{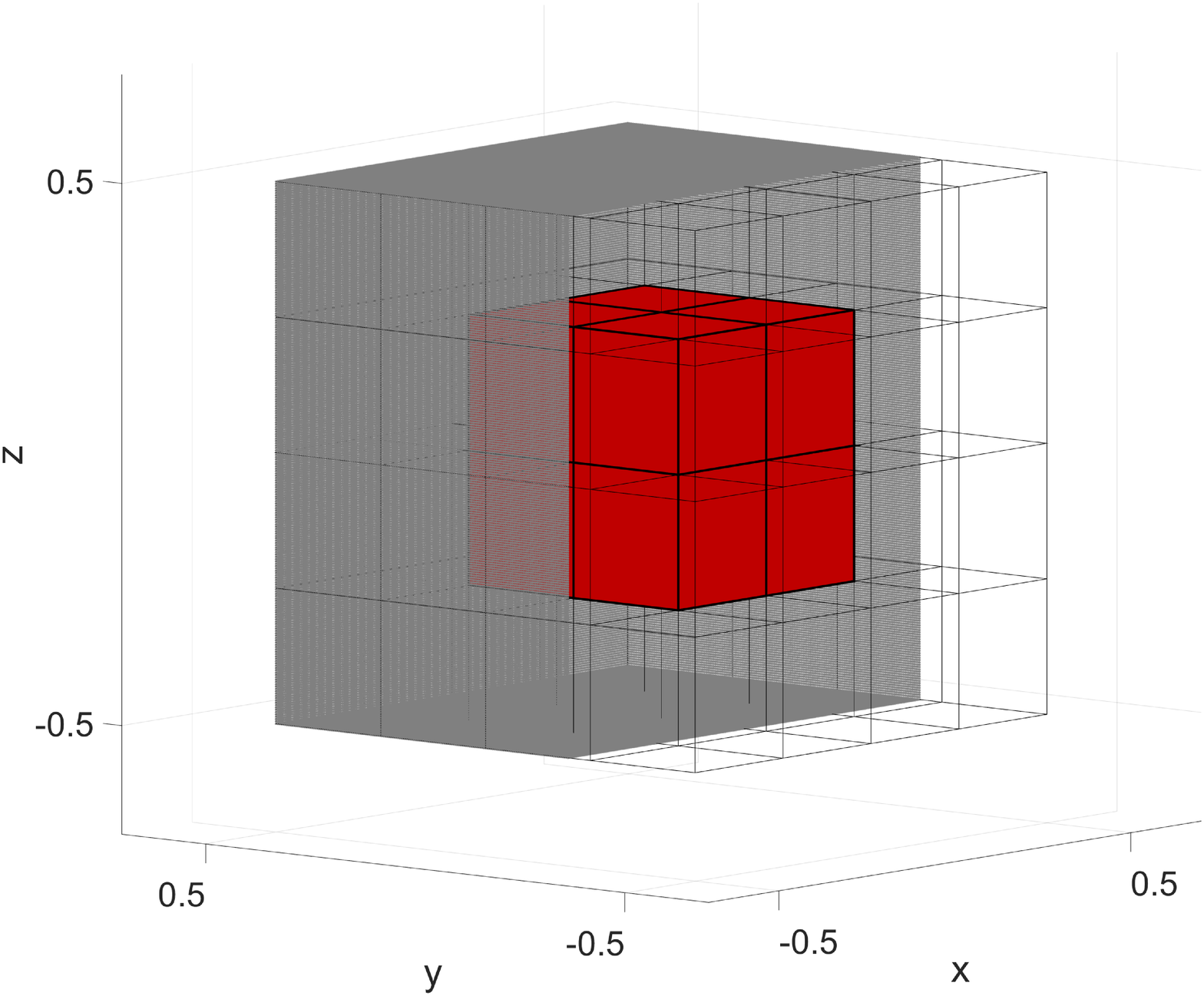}
		\caption{Configuration for the numerical setup in Example \ref{expl:cdelta}: transparent boxes indicate the domain $\Omega=[-\frac{1}{2},\frac{1}{2}]^3$, discretized into a lattice with $\delta = \frac{1}{4}$ (i.e., $N_\delta = 64$ boxes); surrounded by evaluation locations $x'_i$ on $\partial\Omega$, indicated by black dots (all locations with with coordinate value $y<-0.25$ are removed for visualization). The red boxes indicate the domain $[-\frac{1}{4},\frac{1}{4}]^3$, where the function $\funf_a$ with $a=\frac{1}{4}$ is supported.\\}
		\label{fig:numExplLattice}
		\centering
		\includegraphics[trim = 20mm 5mm 20mm 10mm, clip,width=0.7\textwidth]{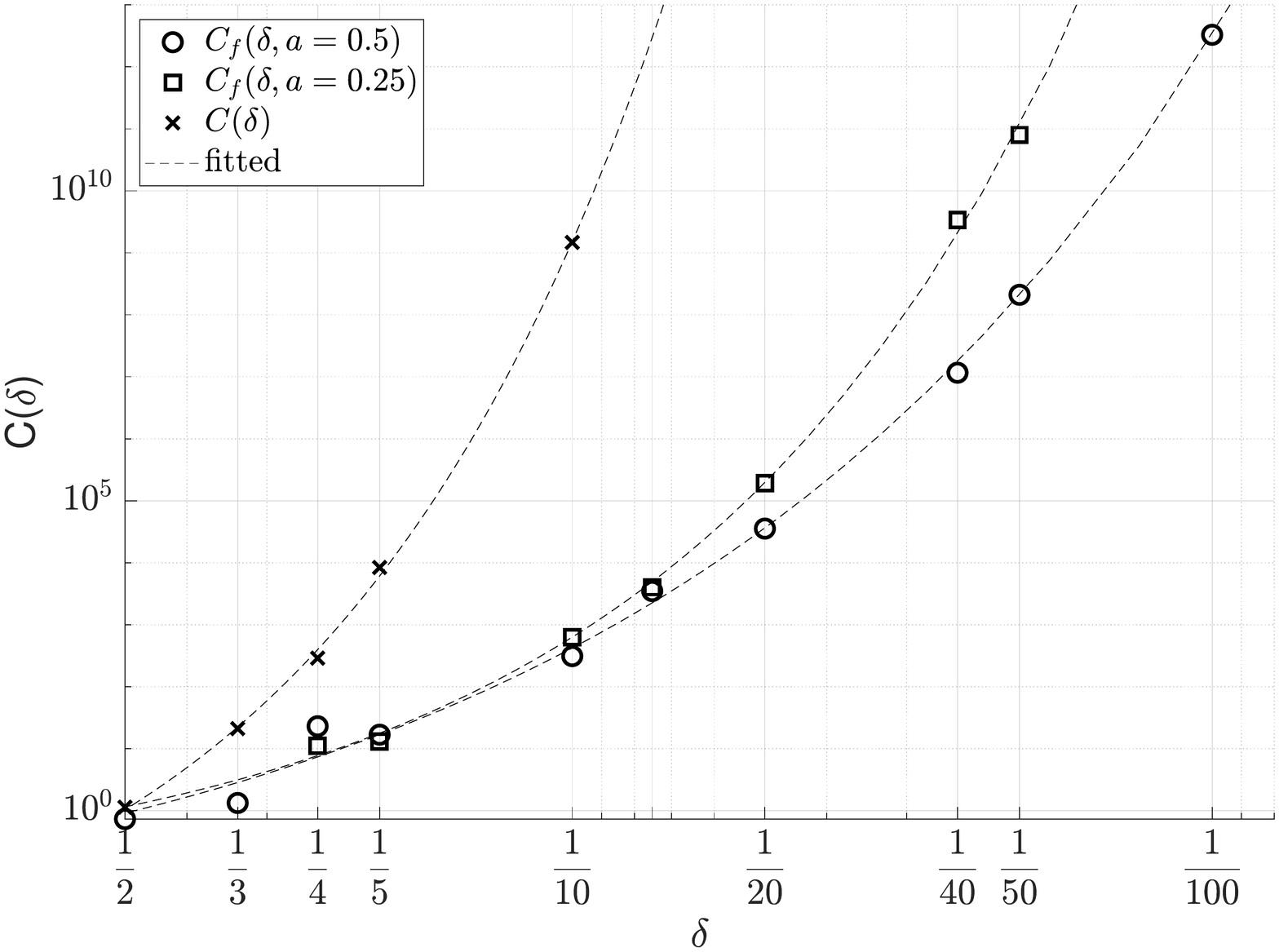}
		\caption{Bounds for the evolution of $C(\delta)$ for varying $\delta$: crosses indicate values obtained by estimation of the operator norm of discretized versions of $\p_{\bol\ot\ol}^{-1}$; circles and cubes indicate bounds obtained by numerical evaluation of the ratio $C_f(\delta)=\|\funf_{\delta}\|_{L^2(\ol,\RR^3)}/\|\p_{\bol\ot\ol}(\funf_{\delta})\|_{L^2(\bol)}$ for $\funf_{\delta}$ as described in Example \ref{expl:cdelta}. Note that $\funf_\delta$ denotes the discretization of the function $\funf=\funf_a$ for two different parameters $a=\frac{1}{2}$ and $a=\frac{1}{4}$, respectively; dashed lines indicate fitted curves of type $\exp(\gamma+\beta\delta^{-\alpha})$, with parameters $\alpha,\beta,\gamma$ as indicated in table \ref{tab:par}.}
		\label{fig:DdeltaDev}
	\end{figure}
	Due to the {known exponential} ill-posedness of inverse potential field problems, it can be expected that the constant $C(\delta)$ from Proposition \ref{prop:stab} grows exponentially with $\delta^{-\alpha}$, for some  fixed $\alpha>0$, as the discretization parameter $\delta$ tends to zero. This will be illustrated numerically in the following example. We use two approaches:
	\begin{itemize}
		\item [(a)]  We choose a function $\funf\in \nabla \sobo \cheese$, for which we know by Theorem \ref{thm:null_spaces} that $\p_{ \cheese^{\outside}\ot\cheese\vpsum}(\funf)=0$, and a box-simple lattice field discretization $\funf_\delta$ of $\funf$, for which we know from Theorem \ref{thm:simple-funtions} that it is not in the nullspace. The ratio $C_f(\delta)=\|\funf_\delta\|_{L^2(\ol,\RR^3)}/\|\p_{\bol\ot\ol}(\funf_\delta)\|_{L^2(\bol)}$ must then be a lower bound of $C(\delta)$, and by construction we know that this ratio must tend to infinity as $\delta$ tends to zero.
		\item[(b)] We discretize the operator $\p_{ \cheese^{\outside}\ot\cheese\vpsum}$ according the used cubic lattice $\Lambda_{\delta}$ and compute the operator norm of its inverse, which should be an appropriate estimate of $C(\delta)$ and clearly be larger than the lower bound from (a). 
	\end{itemize}
\begin{table}
	\caption{Fitting parameters for Figure \ref{fig:DdeltaDev}, with fitted curves of type $\exp(\gamma+\beta\delta^{-\alpha})$.}
	\label{tab:fitting}
	\centering
	\begin{tabular}{c|c|c|c}
		& $\gamma$ & $\beta$ & $\alpha$\\\hline
		$C(\delta)$ & -7.933 & 4.562 & 0.8044 \\
		$C_f(\delta)$, $a=0.5$ & -5.493 & 3.893 & 0.4717 \\
		$C_f(\delta)$, $a=0.25$ & -2.987 & 1.944 & 0.6859
	\end{tabular}\label{tab:par}
\end{table}
	The setup for both approaches is identical: we choose $\cheese=[-\frac{1}{2},\frac{1}{2}]^3$ and subdivide it into a cubic lattice $\Lambda_{\delta,\Omega}$ for $\delta=\frac{1}{2}, \frac{1}{3},\frac{1}{4},\ldots$, with an overall number of $N_\delta=\delta^{-3}$ boxes $\omega_1,\ldots,\omega_{N_{\delta}}$. For a continuous function $\funf$, we define a discretization via $\funf_\delta(x)=\funf(x_{c,i})$, where $x_{c,i}$ is the center of the box $\omega_i$ that contains $x$. The vector $\mathbf{v}_\delta=(\funf(x_{c,1}),\ldots,\funf(x_{c,N_\delta}))\in \RR^{3N_\delta}$ then completely characterizes the function $\funf_\delta\in\RR_{\Lambda_{\delta},\ol}$. The evaluation of $\p_{\cheese^{\outside}\ot\cheese\vpsum}(\funf_\delta)(x'_j)$ at predefined locations $x'_1,\ldots, x'_M\in\Omega^c$ is identical to the matrix vector product $\mathbf{P}_\delta\mathbf{v}_\delta$ with
	\begin{align*}
		\mathbf{P}_\delta=\begin{pmatrix}
			\left(\p_{\cheese^{\outside}\ot\cheese\vpsum}(\chi_{\omega_i}\mathbf{e}_1)(x'_j), \p_{\cheese^{\outside}\ot\cheese\vpsum}(\chi_{\omega_i}\mathbf{e}_2)(x'_j), \p_{\cheese^{\outside}\ot\cheese\vpsum}(\chi_{\omega_i}\mathbf{e}_3)(x'_j)\right)
		\end{pmatrix}_{i=1,\ldots,N_\delta\atop j=1,\ldots,M}\in\RR^{M\times 3N_\delta}
	\end{align*} 
	where $\mathbf{e}_1,\mathbf{e}_2,\mathbf{e}_3$ denote the Cartesian standard basis vectors. The matrix entries can be computed based on the explicit formulas provided in \cite{nagy2009PrismPotential}. An exemplary setup for $N_\delta=64$ and $M\approx330,000$ points on the surface $\partial\Omega$ of the cube $\Omega$ is illustrated in figure \ref{fig:numExplLattice}. 
	
	The exemplary function $\funf\in \nabla \sobo \cheese$ is chosen to be $\funf=\funf_a=\nabla f_a$ with
	\begin{align*} 				f_a(x)=\left\{\begin{array}{ll}\Pi_{k=1}^3\exp\left(-(a^2-(x\cdot\mathbf{e}_k)^2)^{-1}\right), &\max_{k=1,2,3}|x\cdot\mathbf{e}_k|<a,\\0,&\textnormal{else},\end{array}\right.
	\end{align*}
	for two parameters $a=\frac{1}{2}$ and $a=\frac{1}{4}$, respectively. We want to stress that the support of $\funf_a$ is all of $\Omega$ if $a=\frac{1}{2}$, while it is a proper subset of $\Omega$ if $a=\frac{1}{4}$ (indicated by the red cubes in Figure \ref{fig:numExplLattice}). The ratio  $C_f(\delta)=\|\funf_\delta\|_{L^2(\ol,\RR^3)}/\|\p_{\bol\ot\ol}(\funf_\delta)\|_{L^2(\bol)}$ for the discretization can be computed via $\frac{|\partial\Omega|}{M\delta^{d}}\frac{\|\mathbf{v}_\delta\|_{\ell^2}}{\|\mathbf{P}_\delta\mathbf{v}_\delta\|_{\ell^2}}$. The operator norm of the inverse of the discretized version of $\p_{ \cheese^{\outside}\ot\cheese\vpsum}$, which yields $C(\delta)$, is computed via the spectral norm $\delta^d\|(\mathbf{P}_\delta^T\mathbf{P}_\delta)^{-1}\|_2$. In order to compute the inverse matrices $(\mathbf{P}_\delta^T\mathbf{P}_\delta)^{-1}$ to high precision, we used the algorithm from \cite{rump2009IllCondInverse}. 
	
	The outcome of this setup is illustrated in Figure \ref{fig:DdeltaDev} and clearly illustrates the expected exponential growth of $C(\delta)$. Additionally, we see that localization of the function $\funf=\funf_a$ within $\Omega$ can play a role for the size of the bound $C_f(\delta)$: the bound is larger for the choice $a=\frac{1}{4}$, which corresponds to a support of $\funf$ that is a proper subset of $\Omega$. For the case $a=\frac{1}{2}$, corresponding to a support of $\funf$ that covers the full domain $\Omega$, the bound is smaller but still grows exponentially. All computations have been done using Matlab.
\end{expl}

\section{Conclusion}
In this paper, we characterized {non-uniqueness for a fairly general} magnetic inverse problem. We have shown that the {space} of invisible magnetizations {is $\sobo \ol \oplus \sobdivo \ol$ possibly augmented by the subspace of harmonic gradients $\nabla\mathscr{I}_{\olt}$,}  depending on the topology of the problem  (Theorem \ref{thm:null_spaces}).
In applications, this implies {for example} that the inverse problem {has a unique solution} if we can a priori expect the magnetization to be a harmonic gradient. This is for example the case  when the magnetization of the object is induced by an ambient magnetic field, and the susceptibility of the object is constant. But more general {situations are of course} possible, {in which this assumption is not appropriate}. Then{, a harmonic gradient may} not accurately reflect the true magnetization. 

Assuming {instead that the latter is} piecewise constant seems
more appropriate for {many}  applications. The space of piecewise constant vector fields, however, contains invisible vector fields (Examples \ref{exp:triangle} and \ref{expl_ball}). Therefore{,} restricting the magnetization to this space does not make the {solution to the inverse problem unique}. As we {showed}, however, a subspace of piecewise constant vector fields---the space of box-simple fields---yields
{the desired uniqueness property} (Theorem \ref{thm:simple-funtions}). 

The characterization of box-simple vector fields provides us with a different description {of the non-uniqueness phenomenon}. In the case of harmonic gradients, the space that guarantees uniqueness is closed but has a nontrivial orthogonal complement (comprising those spaces of invisible magnetizations mentioned above). In the case of  box-simple vector fields,
{uniqueness does not imply much for the general problem,
	for such fields form
	a space which is dense but not closed in $L^{2}$}.
Therefore every magnetization can be approximated by such a field,
{and non-uniqueness} for the unconstrained inverse problem arises precisely because several box-simple fields can approximate the same $L^{2}$-field. Nevertheless, if a priori considerations assure us that the true magnetization in question is box-simple, our result proves that the reconstruction will be unique.

\paragraph{Acknowledgments.} AK and CG have been partially funded by BMWi (Bundesministerium f\"ur Wirtschaft und Energie) within the joint project 'SYSEXPL -- Systematische Exploration', grant ref. 03EE4002B.

\printbibliography

\end{document}